\documentclass[12pt]{amsart}

\usepackage{tikz-cd}
\usepackage{amsmath}
\usepackage{fourier}
\usepackage{amssymb}
\usepackage{amscd}
\usepackage{amsthm}
\usepackage[centertags]{amsmath}
\usepackage{amsfonts}
\usepackage{newlfont}
\usepackage{graphicx}
\usepackage{amsfonts, amssymb}
\usepackage{mathrsfs}
\usepackage{latexsym}
\usepackage{tikz}
\usepackage{verbatim}
\usepackage[all]{xy}
\usepackage{enumitem}
\usepackage[colorlinks=true,linkcolor=colorref,citecolor=colorcita,urlcolor=colorweb]{hyperref}
\definecolor{colorcita}{RGB}{21,86,130}
\definecolor{colorref}{RGB}{5,10,177}
\definecolor{colorweb}{RGB}{177,6,38}

\numberwithin{subsection}{section}

\newtheorem{theorem}{Theorem}[section]
\newtheorem{problem}[theorem]{Problem}
\newtheorem{proposition}[theorem]{Proposition}
\newtheorem{corollary}[theorem]{Corollary}
\newtheorem{lemma}[theorem]{Lemma}
\newtheorem{definition}[theorem]{Definition}
\theoremstyle{definition}
\newtheorem{remark}[theorem]{Remark}

\theoremstyle{remark}

\usepackage{mathtools}

\DeclareMathOperator{\spa}{span}
\DeclareMathOperator{\id}{\mathrm{id}}

\newcommand\restrict[1]{\raisebox{-.5ex}{ $\vert $}_{#1}}

\makeatletter
\newcommand*\bigcdot{\mathpalette\bigcdot@{.9}}
\newcommand*\bigcdot@[2]{\mathbin{\vcenter{\hbox{\scalebox{#2}{ $\m@th#1\bullet $}}}}}
\makeatother

\begin{document}
\title[Ryll-Wojtaszczyk Formulas for bihomogeneous polynomials]{Ryll-Wojtaszczyk Formulas for bihomogeneous \\ polynomials on the sphere}

\author[Defant]{A.~Defant}
\address{%
Institut f\"{u}r Mathematik,
Carl von Ossietzky Universit\"at,
26111 Oldenburg,
Germany}
\email{defant$@$mathematik.uni-oldenburg.de}

\author[Galicer]{D.~Galicer}   \address{Departamento de Matem\'{a}ticas y Estad\'{\i}stica, Universidad Torcuato Di Tella, Av. Figueroa Alcorta 7350 (1428), Buenos Aires, Argentina and IMAS-CONICET. \tiny{On leave from Departamento de Matematica, FCEyN, Universidad de Buenos Aires.}}  \email{daniel.galicer@utdt.edu}

\author[Mansilla]{M.~Mansilla}
\address{Departamento de Matem\'{a}tica,
Facultad de Cs. Exactas y Naturales, Universidad de Buenos Aires and IAM-CONICET. Saavedra 15 (C1083ACA) C.A.B.A., Argentina}
\email{mmansilla$@$dm.uba.ar}

\author[Masty{\l}o]{M.~Masty{\l}o}
\address{Faculty of Mathematics and Computer Science, Adam Mickiewicz University, Pozna{\'n}, Uniwersytetu
\linebreak
Pozna{\'n}skiego 4,
61-614 Pozna{\'n}, Poland}
\email{mieczyslaw.mastylo$@$amu.edu.pl}

\author[Muro]{S.~Muro}
\address{FCEIA, Universidad Nacional de Rosario and CIFASIS, CONICET, Ocampo  $\& $ Esmeralda, S2000 Rosario, Argentina}
\email{muro$@$cifasis-conicet.gov.ar}

\begin{abstract}
We investigate projection constants for spaces of bihomogeneous harmonic and bihomogeneous polynomials on the unit sphere in finite-dimensional complex Hilbert spaces. Using averaging techniques, we demonstrate that the minimal norm projection aligns with the natural orthogonal projection. This result enables us to establish a connection between these constants and weighted  \linebreak $L_1$-norms of specific Jacobi polynomials. Consequently, we derive explicit bounds, provide practical expressions for computation, and present asymptotically sharp estimates for these constants. Our findings extend the classical Ryll and Wojtaszczyk formula for the projection constant of homogeneous polynomials in finite-dimensional complex Hilbert spaces to the bihomogeneous setting.
\end{abstract}

\date{}

\thanks{The research of the fourth author was supported by the National Science Centre (NCN), Poland, Project 2019/33/B/ST1/00165. The research of the second and fifth authors is additionally supported by 20020220300242BA}

\subjclass[2020]{Primary: 33C55, 33C45, 46B06, 46B07. Secondary: 43A75, 46G25}

\keywords{Projection constants, spherical harmonics, bihomogeneous and finite degree polynomials, Jacobi polynomials.}
\maketitle

%\tableofcontents
\maketitle

\section{Introduction}

This work is devoted to the study of projection constants for spaces of bihomogeneous
harmonic polynomials and bihomogeneous polynomials on Euclidean spheres in
finite-dimensional complex Hilbert spaces. To this end, we employ an abstract framework
that allows us to systematically compute projection constants of finite-dimensional
subspaces of $C(K)$. This framework has previously been applied to a variety of seemingly
unrelated settings, including functions on Boolean cubes \cite{defant2023asymptotic},
Dirichlet series \cite{defant2024projection}, trace class operators
\cite{defant2023integral}, spaces of homogeneous polynomials on real or complex Hilbert spheres
\cite{defant2024minimal, ryll1983homogeneous}. Related questions on projection constants and Bohr-type phenomena
for Banach spaces of analytic polynomials are addressed from a different but close perspective in
\cite{zdefant2025local}.

In particular, \cite{defant2024minimal} provides a thorough analysis of the asymptotic behavior of the projection constants for the spaces  $\mathcal{P}_{\leq d}(\mathbb{S}_{\mathbb{R}}^{n-1}) $,  $\mathcal{P}_d(\mathbb{S}_{\mathbb{R}}^{n-1}) $, and  $\mathcal{H}_d(\mathbb{S}_{\mathbb{R}}^{n-1}) $ on    $\mathbb{S}_{\mathbb{R}}^{n-1}:= \{x \in \mathbb{R}^n : \|x\|_2 = 1\} $, the real Euclidean  sphere. These spaces consist of degree-$d$ polynomials, $d$-homogeneous polynomials, and $d$-homogeneous spherical harmonics, respectively, all equipped with the uniform norm on $\mathbb{S}^{n-1}_{\mathbb{R}}$.

The present article aims to extend this analysis to the complex case, focusing on bihomogeneous settings, i.e., spaces of bihomogeneous polynomials and bihomogeneous harmonic polynomials on the complex Euclidean  sphere  $\mathbb{S}^{n-1}_{\mathbb{C}}:= \{z \in \mathbb{C}^n : \|z\|_2 = 1\} $. For simplicity, we will omit the superscript  ${\mathbb{C}} $ and just use  $\mathbb{S}^{n-1} $ to denote the previously defined set, as most functions in this work are evaluated in the  $n $-dimensional complex space  $\mathbb{C}^n $.

Let us begin by reviewing the necessary definitions. Denote by  $\mathfrak{P}(\mathbb{C}^n) $ the space of all functions  $f\colon \mathbb{C}^n \to \mathbb{C} $ that admit a representation of the form
\begin{equation} \label{writtenf}
f(z) = \sum_{(\alpha,\beta) \in J} c_{(\alpha,\beta)}\,\,z^\alpha\overline{z}^\beta\,,\quad \,\,z \in \mathbb{C}^n\,,
\end{equation}
where  $J $ is a finite index set in  $\mathbb{N}_0^{n} \times \mathbb{N}_0^{n} $, and  $c_{(\alpha,\beta)} \in \mathbb{C} $. The coefficients  $c_{(\alpha,\beta)} $ are uniquely determined, ensuring that the representation of  $f $ in the form above is unique. Indeed, if  $f = 0 $, then applying the differential operator
 $
\frac{\partial^{|\alpha|+|\beta|}}{\partial_{z_1}^{\alpha_1}\ldots \partial_{z_n}^{\alpha_n}\,\partial_{\overline{z}_{1}}^{\beta_1}\ldots \partial_{\overline{z}_{n}}^{\beta_n}}
 $
where, as usual,  $\partial_{z_{j}} := \frac{1}{2}(\partial_{x_{j}}-i\partial_{y_{j}}) $ and  $\partial_{\overline{z}_{j}} := \frac{1}{2}(\partial_{x_{j}}+i\partial_{y_{j}}) $, shows that  $c_{(\alpha,\beta)} = 0 $ for all  $(\alpha, \beta) \neq (0,0) $.

Within this framework, for $p, q\in \mathbb{N}_0$, the subspace of  $(p,q) $-bihomogeneous
polynomials, denoted as  $\mathfrak{P}_{p,q}(\mathbb{C}^n) $, consists of polynomials in  $\mathfrak{P}(\mathbb{C}^n) $ that are
$p $-homogeneous in  $z = (z_j) $ and  $q $-homogeneous in  $\overline{z} = (\overline{z_j}) $. Explicitly, these polynomials take the
form
\begin{equation} \label{(p,q)-writing}
f(z) = \sum_{|\alpha|=p, |\beta|=q} c_{(\alpha,\beta)}\,\,z^\alpha\overline{z}^\beta\,,
\end{equation}
where as usual  $|\gamma|:=\sum \gamma_i $
for  $\gamma = (\gamma_i)\in \mathbb{N}^n_0 $.
A polynomial $ f \in \mathfrak{P}(\mathbb{C}^n) $ is called harmonic if it satisfies $ \Delta f = 0 $, with $ \Delta $ being the
Laplacian
operator, defined as
\[
\Delta = \sum_{j=1}^n \frac{\partial^2}{\partial z_j \partial \overline{z_j}}.
\]

We indicate the space of all harmonic polynomials by  $\mathcal{H}(\mathbb{C}^n) $. The space of  $(p,q) $-bihomogeneous harmonic polynomials is denoted by  $\mathcal{H}_{p,q}(\mathbb{C}^n) $, defined as  $\mathcal{H}_{p,q}(\mathbb{C}^n) := \mathcal{H}(\mathbb{C}^n) \cap \mathfrak{P}_{p,q}(\mathbb{C}^n) $. This space is also known as the space of (solid) harmonics of bidegree $(p,q)$ in $\mathbb{C}^n$.

We focus on the Banach spaces $\mathcal{H}_{p,q}(\mathbb{S}^{n-1}) $ and $ \mathcal{P}_{p,q}(\mathbb{S}^{n-1}) $, consisting of $(p,q)$-bihomogeneous harmonic polynomials and $(p,q)$-bihomogeneous polynomials in $n$ complex variables, respectively, both restricted to the unit sphere $ \mathbb{S}^{n-1} $. Regarded as subspaces of the space of continuous functions $ C(\mathbb{S}^{n-1}) $, equipped with the classical supremum norm, these spaces inherit the structure of Banach spaces.  Furthermore, by homogeneity,  when computing the norm, it is equivalent to take the supremum either over the sphere or over the unit ball of $\ell_2^n(\mathbb{C})$, the $n$-dimensional Hilbert space.
The elements of $ \mathcal{H}_{p,q}(\mathbb{S}^{n-1}) $ are usually called surface harmonics.

Projection constants are essential tools in analyzing the geometric properties of normed spaces. For a Banach space $Y$ and a closed, complemented subspace $X\subset Y$, the projection constant $\lambda(X, Y) $ is defined as the
infimum  of the norms of all bounded linear projections from $Y$ onto $X$. Formally,
\[
\lambda(X, Y) = \inf \big\{ \|P\| : \,\, P \in \mathcal{L}(Y), \,\,\, P|_{X} = \text{id}_{X}\big\}\,.
\]
The absolute projection constant for a Banach space  $ X  $ is given by
\[
\boldsymbol{\lambda}(X) := \sup \, \boldsymbol{\lambda}(I(X),Y)\,,
\]
where the supremum is taken over all possible Banach spaces  $ Y  $ and isometric embeddings  $ I \colon X \to Y  $. In
the case where  $ X  $ is a finite-dimensional Banach space, and  $ X_1  $ is a subspace of a  $ C(K)  $-space that is
isometrically isomorphic to  $X$, we have  (see, e.g., \cite[III.B.5 Theorem]{wojtaszczyk1996banach}):
\begin{align}\label{eq: l_infty inyective}
\boldsymbol{\lambda}(X) = \boldsymbol{\lambda}(X_1, C(K))\,.
\end{align}
This formula implies that finding  $\boldsymbol{\lambda}(X)  $ is the same as determining the norm of a~minimal projection
from $C(K)$ onto $X_1$.

Projections are also of significant importance in approximation theory. When  $Y$ is a Banach space and  $P$ is a~projection
from  $ Y  $ onto a subspace  $X$, the approximation error  $\|y - Py\|_Y  $ for an element $ y \in Y  $
can be bounded by:
\[
\|y - Py\|_Y \leq \|\id_Y - P\| \, \text{dist}(y, X) \leq (1 + \|P\|) \, \text{dist}(y, X)\,,
\]
where  $ \text{dist}(y, X) = \inf\{\|y - x\|_Y \colon  x \in X\}  $. This bound emphasizes the significance of minimizing  $\|P\|$.
A projection  $P_0 \colon Y \to X$ that achieves  $\|P_0\| = \boldsymbol{\lambda}(X, Y)$ is considered a minimal projection onto $X$.

Our primary objective in this work is to analyze the projection constants for the spaces  $\mathcal{H}_{p,q}(\mathbb{S}^{n-1}) $ and
$\mathfrak{P}_{p,q}(\mathbb{S}^{n-1}) $, leveraging the interplay between the degrees  $p $ and  $q $. These investigations naturally
extend classical results such as the renowned formula by Ryll and Wojtaszczyk \cite{ryll1983homogeneous}, which provides an expression
for the projection constants of homogeneous polynomials on complex Hilbert spaces: for all $p\in \mathbb{N}$,
\begin{equation}\label{Ryll-Wojtaszczyk}
\boldsymbol{\lambda}\big(\mathcal{P}_{p}(\mathbb{S}^{n-1})\big)
\,\,=\,\,  \frac{\Gamma(n + p )}{\Gamma(n-1)\Gamma( 1 + p )}\,
\int_0^1 (1 - t)^{n-2} t^{\frac{p}{2}}\,dt  \,\,=\,\,
\frac{\Gamma(n+p) \Gamma(1 + \frac{p}{2})}{\Gamma(1 + p) \Gamma(n + \frac{p}{2})}\,.
\end{equation}

Here, $\mathcal{P}_{p}(\mathbb{S}^{n-1}) $ denotes the space of $p$-homogeneous (analytic) polynomials in  $n$ complex variables,
endowed with the supremum norm on the (complex) sphere  $\mathbb{S}^{n-1} $. Note that in terms of our setting, $\mathcal{P}_{p}(\mathbb{S}^{n-1})$
equals to  $\mathfrak{P}_{p,0}(\mathbb{S}^{n-1}) = \mathcal{H}_{p,0}(\mathbb{S}^{n-1}) $.

By embedding the space $S\in \big\{\mathcal{H}_{p,q}(\mathbb{S}^{n-1}),\,\mathfrak{P}_{p,q}(\mathbb{S}^{n-1})\big\}$
into \( L_2(\mathbb{S}^{n-1}) \), we can naturally define the restriction of the orthogonal projections onto $S$. This results in
projections  from \( C(\mathbb{S}^{n-1}) \) onto \( S \). By employing averaging techniques, we demonstrate that these projections are,
in fact, of minimal norm. This finding provides valuable insight into the geometric structure of the spaces in question.

This enables us to examine the relationship between the projection constants of $\mathcal{H}_{p,q}(\mathbb{S}^{n-1})$ and
$\mathfrak{P}_{p,q}(\mathbb{S}^{n-1})$, and certain weighted $L_1$-norms of specific Jacobi polynomials, as illustrated in Theorems
\ref{jacobi} and \ref{bi-end}. These findings yield an integral formula for the constants of interest.

As a byproduct, we provide upper bounds for the projection constant of these spaces. Moreover, we analyze the correct asymptotic behavior
for the case where the parameters  $(p,q) $ are of the form  $(p, p+d) $ for a fixed distance  $d $ as  $p $ increases (i.e.,  $p $ and  $q $ are equidistant).
Specifically, we prove that in this case the projection constant has asymptotic order $p^{n-\frac{5}{2}}$ and $p^{n-\frac{3}{2}}$, as shown in Propositions $\ref{asymp-harm}$ and $\ref{end}$, respectively. We also determine the correct growth order when  $p $ increases and  $q = 1 $ for the case  $n = 2 $, as stated in Proposition~\ref{nazibude}. Additionally, we present the exact formula for the projection constant when  $(p, q) = (1, 1) $, as given in Proposition~\ref{caso 1-1}, for all dimensions  $n \geq 2 $.

To conclude, we apply the techniques developed to show, in Corollary~\ref{coro: inclusion no compacta}, that any closed subspace $ X $
of $ C(\mathbb{S}^{n-1}) $ that contains
\[
\overline{\operatorname{span}\big\{ \mathcal{H}_{p,q}(\mathbb{S}^{n-1}) : p \in J \big\}},
\]
for an infinite set $ J \subset \mathbb{N} $ and fixed degree $ q $, is not compactly embedded in $ L_1(\mathbb{S}^{n-1}, \sigma_n) $,
where $ \sigma_n $ denotes the normalized rotation invariant measure on  $ \mathbb{S}^{n-1} $.

\section{Brief Preliminaries on Bihomogeneous Polynomial Spaces}

We now present some well-known structural results from the theory of bihomogeneous harmonic and polynomial spaces, which will be useful
for our purposes. These results are well-documented and can be found in various references (see, e.g., \cite{dunkl_xu_2001,folland1975spherical,rudin1980}). An excellent reference for these topics is \cite{klima2004subelliptic}, where the reader can also find further related works.

We include only the statements that are essential for our objectives.

Let us start by recalling the dimension of the spaces  $ \mathcal{H}_{p,q}(\mathbb{S}^{n-1})$. For $n \geq 2$ and  $p, q \geq 0$, we have:
\begin{equation} \label{dimfor-comp}
N_{n,p,q} := \dim \mathcal{H}_{p,q}(\mathbb{S}^{n-1})
= \frac{(n + p + q - 1)(n - 2 + p)! (n - 2 + q)!}{p! q! (n - 1)! (n - 2)!}\,.
\end{equation}

The following two lemmas demonstrate the invariance of the spaces $ \mathfrak{P}_{p,q}(\mathbb{S}^{n-1}) $ and $ \mathcal{H}_{p,q}(\mathbb{S}^{n-1})$
under the action of the unitary group $ \mathcal{U}_n $, as well as the orthogonality between the spaces of bihomogeneous harmonic polynomials.

\smallskip

\begin{lemma} \label{lemma: invariance}
For all $ p, q \geq 0 $, $ f \in \mathfrak{P}_{p,q}(\mathbb{S}^{n-1}) $, and $ U \in \mathcal{U}_n $, we have that $ f \circ U \in \mathfrak{P}_{p,q}(\mathbb{S}^{n-1}) $. The same result holds if\,
$\mathfrak{P}_{p,q}(\mathbb{S}^{n-1}) $  is replaced by $ \mathcal{H}_{p,q}(\mathbb{S}^{n-1}) $.
\end{lemma}

\begin{lemma} \label{ortogonales}
The spaces  $\mathcal{H}_{p,q}(\mathbb{S}^{n-1}) $ and  $\mathcal{H}_{r,s}(\mathbb{S}^{n-1}) $ are orthogonal whenever  $(p,q) \neq (r,s) $.
\end{lemma}

The upcoming theorem is very important; it provides a decomposition of bihomogeneous polynomials into harmonic components (see, e.g.,
\cite[Theorem 1]{bezubik2013spherical} or \cite{bezubik2008spherical,ikeda1968expansion,Koornwinder1972,vilenkin_shapiro_1963}).

\begin{theorem} \label{ludo13A-OKOK}
For $n \geq 2$ and  $p,q \geq 1 $
\begin{equation*} \label{ru5ru5}
\mathfrak{P}_{p,q} (\mathbb{S}^{n-1})=\bigoplus_{j=0}^{\min(p,q)}\mathcal{H}_{p-j, q-j}(\mathbb{S}^{n-1})\,,
\end{equation*}
where the  decomposition is   orthogonal in  $L_2(\mathbb{S}^{n-1}) $.
\end{theorem}

We conclude  this preliminary section with an important density theorem, whose real counterpart can be found in
\cite[Corollary 12.1.3]{rudin1980}.

\begin{theorem} \label{den-matrixA}
 The space $\spa_{p,q \geq 0}\mathcal H_{p,q}(\mathbb{S}^{n-1}) $ is dense in  $C(\mathbb{S}^{n-1}) $.
\end{theorem}

\section{Rudin's Framework: Paving the Way to an Integral Formula} \label{la Rudin} 

We now give an abstract approach  which allows us to express
the projection constants of the spaces under consideration in terms of an $L_1$-norm of a
reproducing kernel, which in the present setting is naturally related to Jacobi polynomials.

The framework outlined here relies on the compact space $K$ having some form of group structure and emerges from isolating techniques developed by Rudin in the study of minimal projections through averaging techniques, as stated in Rudin's works \cite{rudin1962projections, rudin1986new}.  This systematic development may prove beneficial in various other contexts.

We consider  triples  $\big((K,\mu),(G,\mathrm{m}),\varphi\big)$, where
\begin{itemize}
  \item
  $K$ is a compact space together with a Borel probability measure $\mu$ on $K$\,,
  \item
  $G$ is a compact  group with  its Haar measure  $\mathrm{m}$\,,
  \item
  and  $\varphi$ a multiplicative and continuous mapping from $G$ into
  $\text{Hom}(K)$, the space  of all  homeomorphisms from $K$ onto $K$, that is,
  \[
\varphi: G \longrightarrow \text{Hom}(K)\,, \quad g \mapsto \varphi_g\,,
\]
is such that    $\varphi_{gh} = \varphi_g \circ \varphi_h$ for all $g,h\in G $ and   the mapping  $G \times K \ni (g, x) \mapsto \varphi_gx \in K$ is continuous.

\end{itemize}
A subspace $S$ of $C(K)$ is said to be $(G,\varphi)$-invariant
($\varphi$-invariant for short), whenever $ f \circ \varphi_g \in S $ for all $g \in G$ and $f \in S$.

\begin{definition}
\label{def:rudin_triple}
A triple  $\big((K,\mu),(G,\mathrm{m}),\varphi\big) $ is called a Rudin triple if it satisfies the following conditions\:
\begin{itemize}
\item[ $\mathbf{R1} $]
        For all $x, y \in K$  there is $g \in G$  such that $\varphi_g x = y$\,,
    \item[ $\mathbf{R2} $]
        For all $g \in G$ and all $f \in C(K)$ we  have
        $\int_K f \circ \varphi_g \, d\mu = \int_K f \, d\mu $\,,
    \item[ $\mathbf{R3} $] There exists a family $(S_j)_{j\in I}$ of finite-dimensional $\varphi$-invariant subspaces of $C(K)$ whose linear span is dense in $C(K)$.
\end{itemize}

\end{definition}

 The following lemma shows that for any $y \in K$ the pushforward measure under the mapping from $G$ to $K$ given by $G \ni g \mapsto \varphi_g(y)\in K$ coincides with 
$\mu$.

\begin{lemma} \label{greatrudy}
Let $\big((K,\mu),(G,\mathrm{m}),\varphi\big)$ be a Rudin triple. Then for all $f \in C(K)$ and $y \in K$
  \[
  \int_K f d\mu = \int_G  f\big(\varphi_gy\big)\, d\mathrm{m}(g) \,.
  \]
 \end{lemma}

\begin{proof}
Let $y \in K$ be given. Since the mapping 
$G \times K \longrightarrow K\,, \,\, (g,x) \mapsto f(\varphi_gx)$  is continuous,  Fubini's theorem combined with the condition $\mathbf{R2} $ yields the following
\[
\int_K \int_G f\big(\varphi_gx\big)\,  d\mathrm{m}(g)d\mu(x)=\int_G \int_K f\big(\varphi_gx\big)\, d\mu(x)d\mathrm{m}(g) = \int_K f(x)\,d\mu(x)\,.
\]
Fixing $x\in K$, there is $h_x \in G$ such that $\varphi_{h_x}y = x$, and hence by condition $\mathbf{R1} $, the multiplicativity of $\varphi$,
and the invariance of the Haar measure $\mathrm{m}$
\[
\int_G f\big(\varphi_gx\big)\, d\mathrm{m}(g) = \int_G f\big(\varphi_g(\varphi_{h_x}y)\big)\, d\mathrm{m}(g)
 = \int_G f\big(\varphi_{h_x g}y\big)\, d\mathrm{m}(g) = \int_G f\big(\varphi_gy\big)\, d\mathrm{m}(g)\,,
\]
and so
\[
\int_K \int_G f\big(\varphi_gx\big) \, d\mathrm{m}(g)d\mu(x) = \int_G f\big(\varphi_gy\big) \, d\mathrm{m}(g)\,. \qedhere
\]
\end{proof}

\smallskip

\begin{remark} \label{schlau}
 For future reference, it is important to observe that the preceding lemma extends its applicability to a continuous function $f$ on $K$ with values in a Banach space $X$, wherein the conventional integrals are to be replaced by the corresponding Bochner integrals. To establish this, apply the earlier scalar formula to all functions $x^\ast \circ f$, where $x^\ast \in X^\ast$, and invoke the Hahn-Banach theorem for the proof.
\end{remark}

\subsection{Reproducing kernels} \label{inva}

In this section we adapt some well-known facts 
on reproducing kernels to our setting of Rudin triples.
All arguments are rather standard and here repeated for the sake of completeness.

For $L_2(\mu)$-closed subspaces $S$ of $L_2(\mu)$ we denote by $\pi_S$  the orthogonal
projection from $L_2(\mu)$ onto~$S$.
The following lemma, which can be proved easily, shows that the orthogonal projection $\pi_S$  is $G$-equivariant. 

\begin{lemma}\label{lemma: pi_S commutes with L_V}
Let $\big((K,\mu),(G,\mathrm{m}),\varphi\big)$ be a Rudin triple and let $S$ be a 
$\varphi$-invariant, finite dimensional subspace
of \, $C(K)$. Then $\pi_S$ commutes with the action which
$\varphi$ induces on $C(K)$, that is, 
for all $f \in C(K)$ and $g \in G$
\begin{equation*}
    \pi_S(f\circ \varphi_g) = \pi_S(f)\circ \varphi_g\,.
\end{equation*}
\end{lemma}

The following is standard and can essentially also be found in \cite[Lemma 2.4. and Remark 2.5.]{defant2024minimal}.

%\begin{proof}
%Note that  $S^\perp$ is also $\varphi$-invariant, and fix  $ g\in G$ and
%$f\in L_2(\mu)$. Then
%\[
%\text{$(I\!d-\pi_S)(f)\circ \varphi_g\in S^\perp$
%\,\,\,\, and \,\,\,\,$f\circ\varphi_g=\pi_S(f)\circ \varphi_g+(I\!d-\pi_S)(f)\circ \varphi_g$\,,}
%\]
%and hence
%\begin{equation*}
%\pi_S(f\circ \varphi_g)=\pi_S(\pi_S(f)\circ \varphi_g)+\pi_S((I\!d-\pi_S)(f)\circ \varphi_g)=\pi_S(f)\circ \varphi_g\,.
\qedhere
%\end{equation*}

%\end{proof}

\smallskip

\begin{lemma}\label{theorem kernel0}
Let $\big((K,\mu),(G,\mathrm{m}),\varphi\big)$ be a Rudin triple and  $S$ a $\varphi$-invariant, finite dimensional subspace of $C(K)$.
 Then there is a unique continuous function $\mathbf{k}_S: K\times K \to \mathbb C$ such that
for all $f\in L_2(\mu)$  and $x \in K$
\begin{itemize}
\item[(i)]
$
(\pi_Sf)(x)=\big\langle f, \mathbf{k}_S(x, \cdot)\big\rangle_{L_2(\mu)}\,.
$
\end{itemize}
Moreover, the following properties hold{\rm:}
\begin{itemize}
\item[(ii)]
$\mathbf{k}_S(x, \cdot)\in S$\,\,for all $x \in K$\,,
\item[(iii)]
$\mathbf{k}_S(x, y)=\overline{\mathbf{k}_S(y, x)}$ \,\,for all $x,y \in K$\,,
\item[(iv)]
$\mathbf{k}_S(\varphi_gx,\varphi_gy)
=
\mathbf{k}_S(x,y)
$
\,\, for all $g \in G$ and $x  \in K$\,,
\item[(v)]
$\mathbf{k}_S(x,x) = \dim S$ \,\, for all  $x \in K$\,,
\item[(vi)]
$\big(\int_K |\mathbf{k}_S(x, y)|^2 d\mu(y)\big)^{\frac{1}{2}}= \sqrt{\dim S}$\,\, for all  $x \in K$\,.
\item[(vii)] 
  $
 \big\|\pi_S:C(K) \to S\big\|=\int_{K }|\mathbf{k}_S(x,y) |\, d\mu(y) \,$
for all $x\in K$.
\end{itemize}
\end{lemma}

\smallskip

\subsection{Accessability} \label{access}

Given a Rudin triple $\big((K,\mu),
(G, \mathrm{m}),\varphi\big)$, we  say that the 
 $\varphi$-invariant, finite dimensional subspace $S$ of $C(K)$ is accessible whenever the restriction of  the orthogonal projection $\pi_S\restrict{C(K)}$ is the unique projection that in the following sense 'harmonizes' with the group representation which $\varphi$
induces on $K$. Namely, $\pi_S\restrict{C(K)}$ is the unique $G$-equivariant projection from $C(K)$ onto $S$.

\begin{definition}
Let $\big((K,\mu),
(G, \mathrm{m}),\varphi\big)$ be a~Rudin triple.
A finite-dimensional subspace $S$ of $C(K)$ is termed  accessible if it is $\varphi$-invariant and  $\pi_S\restrict{C(K)}$ is the unique projection $\mathbf{Q}$ on $C(K)$ onto $S$ such that
$$\mathbf{Q} (f \circ \varphi_g) = \mathbf{Q}(f) \circ \varphi_g$$
for all $f \in S, g \in G$.
\end{definition}

The following result is one of our main abstract tools. It shows how to compute the projection constant of a subspace $S$ in terms of $L_1$-norm of  the kernel $\mathbf{k}_S(x,\cdot)$ for  $x\in K$.

\begin{theorem} \label{main-ibk}
Let $\big((K,\mu),(G,\mathrm{m}),\varphi\big)$ be a Rudin triple and  $S$  a $\varphi$-invariant, finite dimensional accessible subspace $S$ of $C(K)$. Then
  for all $x \in K$ we have \[
\boldsymbol{\lambda}(S) =  \big\|\pi_S:C(K) \to S\big\|=\int_{K }|\mathbf{k}_S(x,y) |\,d\mu(y) \,.
\]
\end{theorem}

The proof uses the standard technique of averaging projections; see, e.g.~\cite{rudin1962projections} or \cite[Theorem III.B.13]{wojtaszczyk1996banach}, and follows almost verbatim the argument in \cite[Theorem 3.1]{defant2024minimal}.

% \begin{proof}
% Take any projection $\mathbf{Q}$ of $C(K)$ onto $S$, and for every $g\in G$ consider the composition operator $\phi_g \in \mathcal{L}(C(K))$ given by 
% \[
% \phi_g(f) := f\circ \varphi_{g}, \quad\, f\in C(K)\,.
% \]
% Clearly, $\phi_e$ is the identity on $C(K)$ (where $e$ denotes the unity in $G$), and $\phi_{gh} = \phi_g \phi_h$ for all $g, h \in G$. Furthermore,
% \begin{equation*}
%     G \ni g \mapsto \phi_{g^{-1}}\mathbf{Q} \, \phi_g \in \mathcal{L}(C(K))
% \end{equation*}
% is Bochner integrable, as it is bounded by the Banach-Steinhaus theorem and measurable due to its weak measurability. Then  the 
% Bochner integral
% \begin{equation*}\label{equation rudy}
% \int_{G} \phi_{g^{-1}}\mathbf{Q}\,\phi_g\,d\mathrm{m}(g)
% \in \mathcal{L}(C(K))
% \end{equation*}
% is a projection from $C(K)$ that commutes with all $\varphi_g, \, g \in G$. Since $S$ is assumed to be accessible, we have that 
% \begin{equation*}
% \pi_S= \int_{G} \phi_{g^{-1}}\mathbf{Q} \,\phi_g\,d\mathrm{m}(g)\,,
% \end{equation*}
% and consequently  the first equality follows by the triangle inequality (for Bochner integrals). The second claim is then 
% a~consequence of Lemma~\ref{theorem kernel0} (vii).
% \end{proof}

\smallskip

Motivated by the practical challenges in verifying the least restrictive definition of accessibility, we now introduce a stronger notion that implies accessibility, hence its name.  Notably, for the applications we will explore later, this stronger notion is the one we will actually be employing.
\smallskip

\begin{definition}\label{strongaccesibility}
Let $\big((K,\mu),
(G, \mathrm{m}),\varphi\big)$ be a Rudin triple.
We say that a  $\varphi$-invariant, finite dimensional subspace $S$ of  $C(K)$ is  strongly accessible with respect to $x_0 \in K$, whenever  the following holds:
  \begin{equation}\label{point}
    \text{$\forall$    $f \in S$ for which $f \circ \varphi_g = f$ for all $g \in G$ with $\varphi_g x_0 = x_0$
    \, $\exists$ $\lambda \in \mathbb{C}$  $\colon$ \,   $f = \lambda \,\mathbf{k}_{S}(x_0, \cdot) $.}
  \end{equation}
  We call $S$  strongly accessible whenever it is strongly accessible with respect to some  $x_0 \in K$.
\end{definition}

Now we see how this definition implies accessibility.

\begin{theorem} \label{ac-strongacB}
Let $\big((K,\mu),(G,\mathrm{m}),\varphi\big)$ be a Rudin triple and let $S$ be a $\varphi$-invariant, finite dimensional subspace of  $C(K)$.
  If $S$ is strongly accessible, then it is accessible.
\end{theorem}

The proof requires  preparation.

\begin{lemma}\label{ac-lemmaB}
Let $\big((K,\mu),(G,\mathrm{m}),\varphi\big)$ be a Rudin triple and let $H$ and $S$ be both finite dimensional and $\varphi$-invariant subspaces of $C(K)$.
Then  every operator $T\colon H\to  S$ that commutes with the action of $G$, is a scalar multiple of   $\pi_S$, provided $S$ is strongly accessible.

Moreover, if $H$ is  orthogonal to $S$, and $\mathbf{Q}$ is  
a~projection from $H \bigoplus S$ onto $S$ that commutes with the action of
$G$
on $C(K)$,  then $\mathbf{Q}=\pi_{S}\restrict{H \bigoplus S}.$
\end{lemma}

\begin{proof}
Let $S$ be strongly accessible with respect to   $x_0 \in K$. By the assumption on $T$
and Lemma~\ref{theorem kernel0}, (iv)  for  every $g \in G$ with $\varphi_gx_0=x_0$ one has
\[
T \big(\mathbf{k}_H(x_0, \cdot)\big)
\circ \varphi_g
= T \big(\mathbf{k}_H(x_0, \varphi_g (\cdot)) \big)
= T \big(\mathbf{k}_H({\varphi_g^{-1}x_0}, \cdot)  \big) = T \big(\mathbf{k}_H(x_0, \cdot)\big)\,,
\]
and hence by the definition of strong accessability (as in  \eqref{point}) there is $\gamma \in \mathbb{C}$ for which
\begin{equation}\label{firststep}
T  \mathbf{k}_H(x_0, \cdot) = \gamma\,  \mathbf{k}_S(x_0, \cdot)\,.
\end{equation}
On the other hand, we conclude from Lemma~\ref{theorem kernel0} (i),(ii) that for all $h \in H$
\[
h= \pi_H h = \int_{K} h(x) \,\mathbf{k}_H(x, \cdot)\,d\mu(x)\,,
\]
where the integral is meant to be  the Bochner integral of the vector-valued function 
$$
K \longrightarrow C(K), \quad x \mapsto \big[\zeta \mapsto h(x) \mathbf{k}_H(x, \zeta)\big]\,.
$$
 Then by Lemma~\ref{greatrudy} (in the form of Remark~\ref{schlau}) and again Lemma~\ref{theorem kernel0}, (iii) we for all $h \in H$
 have
\begin{align*}
  h =
  &
   \int_{G} h(\varphi_gx_0)
      \mathbf{k}_H(\varphi_gx_0, \cdot)\, d\mathrm{m}(g)
      =
   \int_{G} h(\varphi_gx_0)  \mathbf{k}_H(x_0, \cdot)
    \circ \varphi_g^{-1}\,d\mathrm{m}( (g)\,,
\end{align*}
so that 
\begin{align*}
  Th
  &
  =
     \int_{G} h(\varphi_gx_0) \, T\big(\mathbf{k}_H(x_0, \cdot) \circ \varphi_g^{-1}\big) d\mathrm{m}( (g)
     \\&
          =
     \int_{G} h(\varphi_gx_0)\, T\big(\mathbf{k}_H(x_0, \cdot)\big) \circ \varphi_g^{-1}\,d\mathrm{m}( (g)
     \\&
         =
     \gamma \,
     \int_{G} h(\varphi_gx_0) \, \mathbf{k}_S(x_0, \cdot)
      \circ \varphi_g^{-1}\,d\mathrm{m}( (g)
      \\&
          =
          \gamma\, \int_{G} h(\varphi_gx_0) \, \mathbf{k}_S(\varphi_gx_0, \cdot)\,d\mathrm{m}( (g)
          =
     \gamma \,
     \int_{K} h(x) \,\mathbf{k}_S(x, \cdot)\, d\mu(x)
      = \gamma \, \pi_S(h)
     \,;
            \end{align*}
        here we again use that $T$ commutes with the action of $G$ on $C(K)$, Equation~\eqref{firststep},
Lemma~\ref{greatrudy} (in the form of Remark~\ref{schlau}), and   Lemma~\ref{theorem kernel0}, (i), (iii (iv).
This proves the first part of the proposition.

To see the second assertion, note  that by  the first part of the lemma  we have $\mathbf{Q}\restrict{H}=\gamma \,\pi_S\restrict {H}$ for some
$\gamma  \in \mathbb{C}$. But since by assumption $H\subset S^\perp$, this implies $\mathbf{Q}\restrict{H}= 0=\pi_{S}\restrict {H}$.
 On the other hand,  since $Q$ is a~projection onto $S$, we see that  $\mathbf{Q}\restrict{S}=Id_{S}=\pi_S\restrict {S}$, which finishes
 the proof.
  \end{proof}

\smallskip

\begin{proof}[Proof of Theorem~\ref{ac-strongacB}]
Let $\mathbf{Q}$  be a projection from $C(K)$ onto $S$, which  commutes with the action of
$G$
on $C(K)$.
By  the (density) assumption $\mathbf{R3} $ from Section~\ref{la Rudin}, it suffices to show that for each
$j~\in~I$
\[
\mathbf{Q}\restrict{S_j} = \pi_S\restrict{S_j}\,.
\]
 Fixing $\iota$, we define the subspace
\[
 H := \big\{f-\pi_Sf\,:\,\,f\in S_j \big\} \subset C(K)\,.
\]
Then $H$ is  $\varphi$-invariant; indeed, by the facts that $\pi_S$ commutes with $G$
and $S_j $ is $\varphi$-invariant, for every   $f\in S_j $
and $g \in G$, we have
\[
(f-\pi_S f)\circ \varphi_g=f\circ \varphi_g -\pi_S f\circ \varphi_g= f\circ \varphi_g -\pi_S(f\circ \varphi_g)\in H\,.
\]
 Since $H\perp S$ and $\mathbf{Q}$ commutes with the action of $G$
 on $C(K)$,
 Lemma~\ref{ac-lemmaB} (the second part applied to the restriction  of $\mathbf{Q}$ to $H \oplus S$) shows  that
 $\mathbf{Q}\restrict{H \oplus S}=  \pi_S\restrict{H \oplus S}\,,$
  so  in particular
 $
     \mathbf{Q}\restrict{H}=  \pi_S\restrict{H} = 0\,.
     $
But then
$
   \mathbf{Q}f = \mathbf{Q}(f-\pi_S f)+ \mathbf{Q}(\pi_Sf) =   \pi_Sf
  $
  for  $f \in S_j$,
  which completes the argument.
  \end{proof}

Accessibility is inherited under finite orthogonal sums and under conjugation. The proof follows verbatim the argument in \cite[Proposition 3.2]{defant2024minimal}.

\smallskip

\begin{proposition} \label{one}
Under the general assumptions on $K$ of\, Section~\ref{la Rudin}, a finite $L_2(\mu)$-orthogonal sum $S=\bigoplus S_k$ of accessible subspaces $S_k$ is  again accessible,
  and $\mathbf{k}_S(x, \cdot) = \sum \mathbf{k}_{S_k}(x, \cdot) $ for all $x \in K$.
  In particular, for all $x \in K$
\begin{equation*}\label{integral formula sumA}
\boldsymbol{\lambda}(S_1 \oplus S_2) = \big\|\pi_{S_1} + \pi_{S_2}:C(K) \to S_1 \oplus S_2  \big\|
=
 \int_{K}\big|\mathbf{k}_{S_1}(x, \cdot) +  \mathbf{k}_{S_2}(x, \cdot)\big|\, d\mu
\,.
\end{equation*}
 Moreover,
  if $S$ is accessible, then its conjugate $\overline{S}$ is accessible, and \,
    $\mathbf{k}_{\overline{S}}(x, \cdot) = \overline{\mathbf{k}_{S}(x, \cdot)}$
    for all $x \in K$\,.
\end{proposition}

\subsection{A concrete Rudin triple and its application to our setting} For our purposes, we consider a simple setting. For  $K $, we use the sphere  $\mathbb{S}^{n-1} $ with its  measure  $\sigma_n $. We take the (non-abelian, unimodular) compact group  $G := \mathcal{U}_n $ of all unitary operators on the complex Hilbert space  $\ell^n_2 $, and denote its Haar measure by  $\mathrm{m} $.  We also consider the multiplicative mapping
\[
\varphi \colon \mathcal{U}_n \to \text{Hom}(\mathbb{S}^{n-1}), \quad U \mapsto [z \mapsto Uz]\,.
\]
It is easy to show that the three properties  $\mathbf{R1} $,  $\mathbf{R2} $, and  $\mathbf{R3} $ in Definition \ref{def:rudin_triple} are satisfied, so  $$\big((\mathbb{S}^{n-1},\sigma_n),(\mathcal{U}_n,\mathrm{m}),\varphi\big)$$ is  a Rudin triple.
Indeed, the properties $ \mathbf{R1} $ and $ \mathbf{R2} $ are standard. The density condition $ \mathbf{R3} $ also holds true: from Lemma~\ref{lemma: invariance} and Theorem~\ref{den-matrixA}, we know that all spaces $ \mathcal{H}_{p,q}(\mathbb{S}^{n-1}) $ form a countable family of $ \mathcal{U}_n $-invariant subspaces with a dense union in $ C(\mathbb{S}^{n-1}) $.

As a consequence of Theorem~\ref{main-ibk}, we derive the following important result, which we state explicitly for future references.

\begin{theorem} \label{abstractS}
Let  $S $ be a  $\mathcal U_n $-invariant, finite dimensional  subspace of   $C(\mathbb{S}^{n-1}) $, which is accessible. Then
\[
\boldsymbol{\lambda}(S) =  \big\|\pi_{S} \colon C(\mathbb{S}^{n-1}) \to S\big\|=\int_{\mathbb{S}^{n-1} }\big|\mathbf{k}_S(e_1,z) \big| d\sigma_n (z)\,,
\]
where  $e_1 $ may be replaced by any other vector in  $\mathbb{S}^{n-1} $.
\end{theorem}

Part of the upcoming work involves demonstrating that the subspaces we consider are accessible, and, in addition, we aim to understand how their respective associated reproducing kernels can be expressed.

\bigskip

\section{Jacobi Polynomials as a Foundation for Reproducing Kernels}

Our objective is now to derive explicit formulas for the reproducing kernels associated with the spaces
$\mathcal{H}_{p,q}(\mathbb{S}^{n-1}) $ and  $\mathfrak{P}_{p,q}(\mathbb{S}^{n-1}) $.

Given the parameters  $p,q $, we will show the existence and uniqueness of a~$(p,q) $-bihomogeneous harmonic polynomial
in $n$ variables that is invariant under all orthogonal transformations preserving the vector  $e_1 \in \mathbb{S}^{n-1} $,
and which takes the value 1 at  $e_1$. This result can be compared to the related discussion found in
\cite[Section~2.1.2]{atkinson2012spherical}.

While the results in this section may look familiar, we have opted to include complete proofs for the sake of clarity. It
is important to note that we cannot directly rely on previously established results without adaptation. In many cases, the
arguments require subtle adjustments. Therefore, we believe that providing a thorough and self-contained treatment will be
beneficial, sparing the reader from constantly referring to other sources for missing details.

Before proceeding, we will establish some notation. To simplify our expressions, we as usual define
$p \wedge q := \min\{p, q\}$ and  $p \vee q := \max\{p, q\}$.

Let  $\mathcal{U}_n(\xi) $ denote the set of unitary transformations that fix the vector  $\xi \in \mathbb{S}^{n-1} $.
Given  $n\in \mathbb{N} $, $p,q \in \mathbb{N}_0,$ we define the  $(p,q) $-Legendre polynomial  $L_{n,p,q} \colon \mathbb{C}^n \to \mathbb{C}$ by
\begin{equation*}
L_{n,p,q}(z) :=  \sum_{j=0}^{p \wedge q} c_j(n,p,q)\, z_1^{p-j}\overline{z}_1^{q-j}\,\|(z_2, \ldots, z_{n})\|_2^{2j}\,,
\end{equation*}
where for
\begin{equation}\label{unique-harm-comp}
c_j(n,p,q) =   \frac{(-1)^j p! \,q! \,(n-2)!}{j!\,(p-j)!\,(q-j)! \,(j+n-2)!}\,, \quad  j = 0, \ldots, p \wedge q\,.
\end{equation}

The name of the polynomial may not appear in the literature, but it is based on the real case where one encounters the classical
Legendre polynomial (see \cite[Section~2.1.2]{atkinson2012spherical}).

A very important characterization of the $(p,q)$-Legendre polynomial is presented here. The proof of this fact relies heavily
on \cite[Theorem 12.2.6]{{rudin1980}}.

\smallskip
\begin{proposition} \label{propleg-comp}
The polynomial  $L_{n,p,q} \colon \mathbb{C}^n \to \mathbb{C} $ is the unique  polynomial
$g \in \mathcal{H}_{p,q}(\mathbb{C}^n) $ which fulfills the following two properties\,:
\begin{itemize}
\item[(a)]
$g\circ A = g $ for all  $A\in \mathcal{U}_n(e_1)$\,,
\item[(b)]
$g(e_1) = 1 $\,.
\end{itemize}
\end{proposition}

\begin{proof}
Take  $g \in \mathcal{H}_{p,q}(\mathbb{C}^n) $ which satisfies  $(a) $ and  $(b) $,  and consider its representation
\[
g(z) = \sum_{|\alpha|=p, |\beta|=q} c_{(\alpha,\beta)}
z^\alpha \bar{z}^\beta , \quad z \in \mathbb{C}^n\,.
\]
We write  $z = (z_1,z') \in \mathbb{C} \times \mathbb{C}^{n-1} $ for   $z \in \mathbb{C}^n $, and note that
every  $n\times n $-matrix  $U $ belongs to  $\mathcal{U}_n(e_1) $ if and only if it has the form
\begin{equation}\label{0a0}
  A=
\begin{pmatrix}
  1 & 0  \\
  0 & A' \\
 \end{pmatrix}
\end{equation}
with  $A' \in \mathcal{U}_{n-1} $.  Fix  $z_1 \in \mathbb{R} $,  and consider the polynomial
\[
h(z') = g(z_1, z'),  \quad z' \in \mathbb{C}^{n-1}.
\]
Then for all  $A' \in \mathcal{U}_{n-1} $ and  $z' \in \mathbb{C}^{n-1} $
\[
 h(A'z')= g(z_1,A'z' ) = g(A(z_1,z') ) = g(z_1,z' )=  h(z')\,,
\]
and in particular  $h(z') = h(-z') $. Consequently,
\[
h(t, 0, \ldots, 0) = \sum_{j=0}^{\lfloor \frac{p+q}{2}\rfloor} a_j(z_1) t^{2j},  \quad t \in \mathbb{R}\,.
\]
Choosing for  $z' \in \mathbb{C}^{n-1} $ some  $A' \in \mathcal{U}_{n-1} $ such that  $A' z' = \|z'\|_2 e_1 $,
we obtain
\begin{align*}
  \sum_{|\alpha|=p,|\beta|=q} &c_{(\alpha,\beta)}
    z_1^{\alpha_1}\bar{z}_1^{\beta_1} (z')^{(\alpha_2, \ldots, \alpha_{n})}
  (\overline{z}')^{(\beta_2, \ldots, \beta_{n})}  =g(z_1,z')
\\&
 = h(z') = h(\|z'\|_2 e_1) = \sum_{j=0}^{\lfloor \frac{p+q}{2}\rfloor} a_j(z_1) \|z'\|_2^{2j}\,.
\end{align*}
If we now compare coefficients,  then for all  $z \in \mathbb{C}^n $

\[
g(z) = \sum_{j=0}^{m } c_j z_1^{p-j}\overline{z}_1^{q-j} \|z'\|_2^{2j}\,,
\]

where  $m=p \wedge q $ and so it remains to verify that  the coefficients  $c_j=c_n(n,p,q) $ are those
from \eqref{unique-harm-comp}.  Indeed,  applying the Laplacian  $\triangle $ to both sides of the
preceding equality, we see that for all  $z \in \mathbb{C}^n $
\[
\triangle g(z) =
\sum_{j=0}^{ m  -1} b_j z_1^{p-1- j}\overline{z}_1^{q-1-j} \|z'\|_2^{2j}\,,
\]
where
\[
b_j = (p-j) (q-j) c_j  + (j+1) (j+n-1) c_{j+1}\,, \quad  j= 0, 1, \dots, m-1\,.
\]
But  $\triangle g = 0 $, so
\[
c_{j+1} =  \frac{-(p-j) (q-j)}{(j+1) (j+n-1)}\, c_j \,, \quad  j= 0, 1, \dots, m-1\,.
\]
Since $g(e_1)=1$ implies that $c_0=1,$ we have \eqref{unique-harm-comp}, as desired.
Conversely, it is clear that the polynomial
\[
g(z) = \sum_{j=0}^{p \wedge q } c_j z_1^{p-j}\overline{z}_1^{q-j} \|z'\|_2^{2j}
\]
is a $(p,q) $-bihomogeneous harmonic polynomial, as deduced from the definition of the coefficients  $(c_j)_j$
and the reasoning above. Furthermore, since  $c_0=1 $, we have  $g(e_1)=1 $. The invariance under the subgroup
$\mathcal{U}_n(e_1) $ is obvious.
\end{proof}

To derive key consequences from the previous proposition, we now introduce a useful identification. Recall that
$\mathbb{D} $ represents the complex open unit disk. Given  $z \in \mathbb{S}^{n-1}$, we have
\[
L_{n,p,q}^{\diamond} (\langle z, e_1 \rangle) =  L_{n,p,q}(z)\,,
\]
where the function  $L_{n,p,q}^{\diamond}: \mathbb{D} \to \mathbb{R} $, for  $w \in \mathbb{D} $, is defined as
\begin{equation} \label{repro}
  L_{n,p,q}^{\diamond}(w) :=
    \sum_{j=0}^{p\wedge q} c_j(n,p,q) \,w^{p-j} \overline{w}^{q-j}
    (1 - |w|^2)^{j}.
\end{equation}

  \smallskip

\begin{corollary} \label{lemma1-comp}
Let  $f \in \mathcal{H}_{p,q}(\mathbb{S}^{n-1}) $ and  $\xi \in \mathbb{S}^{n-1} $. Then the following 
conditions are equivalent:
\begin{itemize}
\item[{\rm(1)}]
$f \circ A = f $ for all  $A \in \mathcal{U}_n(\xi)$\,.
\item[{\rm(2)}]
$f = f(\xi)\, L^{\diamond}_{n,p,q}(\langle\pmb{\cdot},\xi \rangle )$\,.
\end{itemize}
\end{corollary}

\begin{proof}
Using the properties of  $L_{n,p,q} $ isolated in Proposition~\ref{propleg-comp}, the implication  $(2)\Rightarrow(1) $ is obvious.
In order to  check that   $(1)\Rightarrow(2) $, we first assume that  $\xi = e_1 $, and  choose
$g \in \mathcal{H}_{p,q}(\mathbb{C}^n) $ such that  $g|_{\mathbb{S}^{n-1}} = f $. By assumption for all  $A \in \mathcal{U}_n(e_1) $ and  $z \in \mathbb{C}^n $
\[
g(Az) = \|Az\|^{p+q}_2 f \Big(\frac{Az}{\|Az\|_2}\Big)
=
\|z\|^{p+q}_2 (f\circ A) \Big(\frac{z}{\|z\|_2}\Big)
=
\|z\|^{p+q}_2 f \Big(\frac{z}{\|z\|_2}\Big)
= g(z)\,.
\]
Hence, by the uniqueness properties of   $L_{n,p,q} $ from Proposition~\ref{propleg-comp} there is some constant  $c >0 $ such that
$g = c L_{n,p,q} $ on  $\mathbb{C}^n $, and inserting here  $e_1 $ we see that  $c = g(e_1)=f(e_1) $. Then, for all  $\eta \in \mathbb{S}^{n-1} $,
we get
\begin{equation*}
f(\eta) = g( \eta) = c L_{n,p,q}(\eta) = f(e_1) L^{\diamond}_{n,p,q}(\langle\eta, e_1\rangle)\,.
\end{equation*}
Finally, we take an arbitrary  $\xi \in \mathbb{S}^{n-1} $,  and choose   $B \in \mathcal{U}_n $ such that
$Be_1 = \xi $.  Since  $B \circ A \circ B^{-1} \in \mathcal{U}_n(\xi) $ for all
$A \in \mathcal{U}_n(e_1) $, by  $(1) $ we have  $(f \circ B) \circ A \circ B^{-1} = f $ for all
$A \in \mathcal{U}_n(e_1) $, and hence  $(f \circ B) \circ A = f \circ B $ for all such  $A $. Consequently,
$f \circ B = f (B(e_1)) L^{\diamond}_{n,p,q}(\langle\pmb{\cdot},e_1\rangle) $,  or equivalently
$f = f(\xi) L^{\diamond}_{n,p,q}(\langle\pmb{\cdot},\xi \rangle) $.
\end{proof}

\smallskip

\smallskip
In the literature, $\mathbf{k}_{\mathcal{H}_{p,q}(\mathbb{S}^{n-1})}(\xi, \cdot)$ is often referred to as a zonal spherical harmonic of
bidegree $(p,q)$ with pole at $\xi$, and is sometimes denoted by $Y^{(p,q)}_\xi$. To maintain consistency, we will retain our current
notation. As a consequence of the previous result, we obtain deeper insight into the structure of this reproducing kernel.

\begin{lemma} \label{lemma: repro}
For $n \geq 2 $ and  $p,q \geq 0 $ we have:
\begin{align*}
&
\mathbf{k}_{\mathcal{H}_{p,q}(\mathbb{S}^{n-1})}(e_1, \eta) = N_{n,p,q}\,\,L^{\diamond}_{n,p,q}(\eta_1)\,.
\end{align*}
\end{lemma}

 \begin{proof}
 We consider  $S = \mathcal{H}_{p,q}(\mathbb{S}^{n-1}) $. Then by Lemma~\ref{theorem kernel0} (iv)  for all $A \in \mathcal{U}_n(\xi)$
 and  $y \in \mathbb{S}^{n-1} $,  we have
   \begin{equation*}
     \mathbf{k}_S(\xi, y) = \mathbf{k}_S(A\xi, Ay) = \mathbf{k}_S(\xi, Ay)\,,
   \end{equation*}
   so that by Corollary~\ref{lemma1-comp}
   \begin{equation*}
     \mathbf{k}_S(\xi, \cdot) = \mathbf{k}_S(\xi, \xi) L^{\diamond}_{n,p,q}(\langle\pmb{\cdot},\xi\rangle)\,.
   \end{equation*}
   Then Lemma~\ref{theorem kernel0}, (i) and (v) complete the argument.
 \end{proof}

We now deal with the accessibility of the spaces of interest.

\begin{lemma}    \label{accessibility-comp}
Let $n \geq 2$ and $p, q \geq 0$. Then the space $\mathcal{H}_{p,q}(\mathbb{S}^{n-1})$ is strongly accessible. In particular,
$\mathfrak{P}_{p,q}(\mathbb{S}^{n-1})$ is also accessible, and we have the following representation for its reproducing kernel\,:
\[
\mathbf{k}_{\mathfrak{P}_{p,q}(\mathbb{S}^{n-1})}(e_1,\eta) = \sum_{j=0}^{p \wedge q} N_{n,p-j,q-j} L^{\diamond}_{n,p-j, q-j}(\eta_1)\,.
\]
\end{lemma}

\begin{proof}
The strong accessibility of the space $\mathcal{H}_{p,q}(\mathbb{S}^{n-1})$ can be established through a~careful analysis of Definition
\ref{strongaccesibility} in conjunction with Corollary~\ref{lemma1-comp}. Recalling from Theorem~\ref{ludo13A-OKOK} the orthogonal
decomposition of $\mathfrak{P}_{p,q}(\mathbb{S}^{n-1})$, we conclude from Proposition~\ref{one} the accessibility of
$\mathfrak{P}_{p,q}(\mathbb{S}^{n-1})$. We moreover obtain  the representation of its reproducing kernel:
% Since accessibility is preserved under finite orthogonal sums, as noted in Proposition~\ref{one}, we can deduce the accessibility of $\mathfrak{P}_{p,q}(\mathbb{S}^{n-1})$. Moreover, by recalling Theorem~\ref{ludo13A-OKOK}, we obtain the representation of its reproducing kernel:
\[
\mathbf{k}_{\mathfrak{P}_{p,q}(\mathbb{S}^{n-1})}(e_1,\eta) = \sum_{j=0}^{p \wedge q} \mathbf{k}_{\mathcal{H}_{p-j, q-j}(\mathbb{S}^{n-1})}(e_1,\eta)\,.
\]
The conclusion now follows from Lemma~\ref{lemma: repro}.
\end{proof}

Combining Theorem~\ref{abstractS} with Lemmas~\ref{lemma: repro} and~\ref{accessibility-comp}, we arrive at the following theorem, which
will later serve as the foundation to obtain concrete results.

\smallskip

\begin{theorem}  \label{top}
For $n\geq 2$ and $p,q \geq 0 $ we have the integral formulas:
\begin{align*}
\boldsymbol{\lambda}\big(\mathcal{H}_{p,q}(\mathbb{S}^{n-1})\big) = N_{n,p,q}\,\int_{\mathbb{S}^{n-1}}\,
\Big|  \,\,L^{\diamond}_{n,p,q}(\eta_1)\Big| \,d\sigma_n(\eta)
\end{align*}
and
\begin{align*}
&
  \boldsymbol{\lambda}\big(\mathfrak{P}_{p,q}(\mathbb{S}^{n-1})\big)
 =
\int_{\mathbb{S}^{n-1}}\,
 \Big| \sum_{j=0}^{p \wedge q} \,N_{n,p-j,q-j}\,L^{\diamond}_{n,p-j, q-j}(\eta_1)\Big| \,d\sigma_n(\eta)\,.
\end{align*}
\end{theorem}

We now aim to provide a precise and meticulous description of the reproducing kernels, highlighting the connection between these and certain Jacobi polynomials.
For this we recall the definition of this class. For parameters  $\alpha, \beta > -1 $, we denote the sequence of degree-$d $ Jacobi polynomials defined on the interval  $[-1,1] $ as  $\big(P_d^{\alpha, \beta}\big)_{d \in \mathbb{N}_0} $. These polynomials are characterized by their  orthogonality with respect to the weighted inner product given by
\[
\langle P, Q \rangle = \int_{-1}^{1} P(t) Q(t) \, (1-t)^\alpha (1+t)^\beta \, dt\,.
\]
Additionally, they are normalized such that
\[
P_d^{\alpha, \beta}(1) = \binom{d+\alpha}{d}\,.
\]
There are several explicit formulas  for these polynomials; one notable expression is provided by the Rodrigues formula:
\begin{equation}\label{rodjac}
  P_d^{\alpha, \beta}(t) = \frac{(-1)^d}{2^d d!}
  (1-t)^{-\alpha} (1+t)^{-\beta}
  \left(\frac{\partial}{\partial t}\right)^d \left[(1-t)^{\alpha+d} (1+t)^{\beta+d}\right]\,.
\end{equation}

The following result is familiar and can be found for example in \cite[Theorem 3.3.]{Koornwinder1972}. We give an independently
interesting proof which is based on the representation of  $L^{\diamond}_{n,p,q} $ from Equation~\eqref{repro}. We need to recall
the well-known formula for integration on spheres, which will be used frequently (see, e.g., \cite[1.4.5. (2)]{rudin1980}):
for every  $f \in C(\overline{\mathbb{D}}) $, we have
\begin{equation}\label{integration}
\int_{\mathbb{S}^{n-1}} f(\langle \eta , e_1 \rangle) \, d\sigma_n(\eta) =
\frac{n-1}{\pi} \int_{0}^{1} \int_{-\pi}^{\pi} (1-r^2)^{n-2} f(r e^{i\theta}) \, r \, d\theta \,dr \,.
\end{equation}

\begin{theorem} \label{Koornwinder}
  Let  $p,q \geq 0 $ and  $n \geq 2 $. Then for all  $z = r e^{i \theta} $ with  $r\geq0 $,
  \[
  L^{\diamond}_{n,p,q}(z) = r^{|p-q|}  e^{i(p-q)\theta}
  \,\,
  \frac{(p\wedge q)!(n-2)!}{((p\wedge q) + (n-2))!}
  \,\,
  P_{p\wedge q}^{n-2,|p-q|}(2r^2 -1)\,.
  \]
\end{theorem}

\begin{proof}
Fixing  $n \geq 2$, we   in the following  divide the proof  into two steps.  In the first step we prove that for all  $z = re^{i\theta} \in \overline{\mathbb{D}} $  and  $p,q \ge 0$
  \[
  L^{\diamond}_{n,p,q}(z) = r^{|p-q|}  e^{i(p-q)\theta}
  \,\,
  Q_{p,q}(2r^2 -1) \,,
  \]
 where   $Q_{p,q}: \mathbb{R} \to \mathbb{R} $  is a polynomial  of degree
  less than or equal to $ p \wedge q  $.
  To see this, assume  without loss of generality  that  $p \wedge q = p $.
Then by Equation~\eqref{repro}
  \begin{equation*}
  L_{n,p,q}^{\diamond} (z) =  r^{|p-q|}  e^{i (p-q)\theta}
    \sum_{j=0}^{p}c_j \,r^{2(p-j)}
    (1-r^2)^j  \,,
 \end{equation*}
 and for each  $0 \leq j \leq p $
 \begin{align*}
   &
   r^{2(p-j)} = 2^{-(p-j)}((2r^2 -1)+1)^{p-j} = u_j(2r^2 -1)
   \\&
   (1-r^2)^j
    = (-1)^j 2^{-j}(2(r^2-1))^j = (-1)^j 2^{-j}((2r^2-1)-1))^j =v_j(2r^2 -1)\,,
     \end{align*}
 where  $u_j(x) = 2^{-(p-j)}(x+1)^{p-j} $ is a polynomial of degree  $ p-j $
 and  $v_j(x) = (-1)^j 2^{-j}(x-1)^{j} $ a~polynomial of degree  $ j $.  This clearly proves the claim.

  In the second step we  check that   for all  $r \in [0,1] $ and $p,q \ge 0 $
  \[
  Q_{p,q}(2r^2 -1) = \frac{1}{P_{p\wedge q}^{n-2,|p-q|}(1)}\,\, P_{p\wedge q}^{n-2,|p-q|}(2r^2 -1)\,.
  \]
  To do so, we fix some  arbitrary $k \in \mathbb{Z} $,  and
  show that the preceding  equality holds
  for all  $p,q \ge 0 $ with $k = p-q$.  The idea is to
  prove that in this case all  polynomials $Q_{p,q}$ are
  orthogonal with respect to the weight function
 $ (1-\pmb{\cdot})^{n-2}(1+ \pmb{\cdot})^{|k|}$  on  $[-1,1]$ , which then implies the conclusion using the uniqueness of the Jacobi polynomials.

 More precisely,  assume that for each  $p,q,p',q' \ge 0$
 such that $k = p-q=p'-q'$ and
 $\ell = p\wedge q \neq \ell' = p'\wedge q' $,
 we have,
  \begin{align} \label{koorwin}
    \int_{-1}^{1}  (1-s)^{n-2}(1+s)^{|k|} Q_{p,q}(2s^2 -1)  Q_{p',q'}(2s^2 -1) ds = 0\,.
  \end{align}
Then, by the uniqueness properties of Jacobi polynomials, there exists a constant $c \in \mathbb{R}$ such that
\[
Q_{p,q}(2s^2 -1) = c \, P_{p \wedge q}^{n-2, |k|}(s), \quad s \in [-1,1].
\]
Moreover, since $Q_{p,q}(1) = L_{n,p,q}^{\diamond}(1) = 1$, it follows that
\[
c = \frac{1}{P_{p \wedge q}^{n-2, |p-q|}(1)}.
\]
  This would complete the proof of our claim.

It remains to prove Equation~\eqref{koorwin}. We use Equation~\eqref{integration} to see that
  \begin{align*}
    &
  \int_{\mathbb{S}^{n-1}}
    L_{n,p,q}^{\diamond} (\langle e_1, \eta \rangle)  \overline{
L_{n,p',q'}^{\diamond}(\langle e_1, \eta \rangle)}
    \,d\sigma_n(\eta)
    \\&
  =
  \frac{n-1}{\pi}\int_{0}^{1} \int_{-\pi}^{\pi}
    r^{|p-q|+|p'-q'|}  e^{i (p-q)\theta} e^{-i (p'-q')\theta} (1-r^2)^{n-2} r
    Q_{p,q} (2r^2 -1)  Q_{p',q'} (2r^2 -1)
  d\theta dr
  \\&
  =
  (n-1)\int_{0}^{1}
    (r^2)^{|k|}  (1-r^2)^{n-2} r
    Q_{p,q} (2r^2 -1)  Q_{p',q'} (2r^2 -1)
  dr
  \\&
  =
  2^{-|k|-(n-2)} (n-1)  \int_{-1}^{1} (1+s)^{|k|} (1-s)^{n-2} Q_{p,q}(s) Q_{p',q'} (s) ds\,.
    \end{align*}
where we have used the change of variables  $s=2r^2-1 $ to get the last equality.
    But
      \begin{align*}
      \int_{\mathbb{S}^{n-1}}
    L_{n,p,q}^{\diamond} (\langle e_1, \eta \rangle)  \overline{
L_{n,p',q'}^{\diamond} (\langle e_1, \eta \rangle)}
    \,d\sigma_n(\eta) = 0\,,
    \end{align*}
  since by Proposition~\ref{propleg-comp} we have that
   $L_{n,u,v}^{\diamond} (\langle e_1, \cdot \rangle) \in \mathcal{H}_{u,v}(\mathbb{S}^{n-1})  $ for all  $u,v \geq 0 $, and by Lemma~\ref{lemma: invariance}  that all spaces
   $\mathcal{H}_{u,v}(\mathbb{S}^{n-1})  $ are pairwise orthogonal. This  completes the argument.    \end{proof}

\section{Integral Formulas}

We now shift our focus to some explicit integral formulas for the projection constant. The following observation will serve as our guide.

\begin{remark} \label{RW}
The   Ryll-Wojtaszczyk formula from \eqref{Ryll-Wojtaszczyk} states that for all  $p \in \mathbb{N}, $
\begin{equation*}
  \boldsymbol{\lambda}\big(\mathcal{P}_{p}(\mathbb{S}^{n-1})\big)
 \,\,=\,\,  \frac{\Gamma(n + p )}{\Gamma(n-1)\Gamma( 1 + p )}\,
  \int_0^1 (1 - t)^{n-2} t^{\frac{p}{2}}\,dt  \,\,=\,\,
\frac{\Gamma(n+p) \Gamma(1 + \frac{p}{2})}{\Gamma(1 + p) \Gamma(n + \frac{p}{2})} \,.
\end{equation*}
The previous result asserts that for $q=0$,
\begin{equation*} \label{RN}
 \boldsymbol{\lambda}\big(\mathfrak{P}_{p,q}(\mathbb{S}^{n-1})\big) =\boldsymbol{\lambda}\big(\mathcal{H}_{p,q}(\mathbb{S}^{n-1})\big)
= \frac{\Gamma(n + p \vee q )}{\Gamma(n-1)\Gamma( 1 + p \vee q )}\,\,  \int_0^1 (1 - t)^{n-2} t^{\frac{p}{2}}\,dt\,.
\end{equation*}
  \end{remark}

We aim to generalize this equation to  $\mathcal{H}_{p,q}(\mathbb{S}^{n-1}) $ and  $\mathfrak{P}_{p,q}(\mathbb{S}^{n-1}) $ for  $p, q \ge 0 $.

Now that we have the integral representations and a clear understanding of how to express the kernels, we present explicit expressions for the projection constants of $\mathfrak{P}_{p,q}(\mathbb{S}^{n-1})$ and $\mathcal{H}_{p,q}(\mathbb{S}^{n-1})$ in terms of weighted  $ L_1 $-norms, where the weights are determined by specific Jacobi polynomials.

Let us first address the spherical harmonics.

\begin{theorem} \label{jacobi}
   For    $n \geq 2 $
 and  $p,q \geq 0 $
\begin{equation*}
  \boldsymbol{\lambda}\big(\mathcal{H}_{p,q}(\mathbb{S}^{n-1})\big)
  = c(p,q,n)\,\, \int_{0}^{1} (1-t)^{n-2} t^{\frac{|p-q|}{2}} \,\big| P_{p\wedge q}^{n-2,|p-q|}(2t -1)\big| dt\,,
\end{equation*}
where
\begin{align*}
  c(p,q,n)
    =
   (n+p+q-1)\,
  \frac{\Gamma(n+ p \vee q -1)}{ \Gamma(n-1) \Gamma(1+ p \vee q)}\,,
\end{align*}
and $P_{p\wedge q}^{n-2,|p-q|}$ stands for the Jacobi polynomial of degree $p\wedge q$ as defined in, e.g., \eqref{rodjac}.
Equivalently,
\begin{equation*}
  \boldsymbol{\lambda}\big(\mathcal{H}_{p,q}(\mathbb{S}^{n-1})\big)
  = \frac{c(p,q,n)}{2^{n + \frac{|p-q|}{2}}}\,\, \int_{-1}^{1} (1-s)^{n-2} (1+s)^{\frac{|p-q|}{2}} \,\big| P_{p\wedge q}^{n-2,|p-q|}(s)\big| ds\,.
\end{equation*}
  \end{theorem}

Note that the case  $p \wedge q= 0 $ is precisely the Ryll-Wojtaszczyk result  described in Remark~\ref{RW}.

\begin{proof}
From  Lemma~\ref{lemma: repro} and Theorem~\ref{Koornwinder}, we deduce that  for every  $\eta \in \mathbb{S}^{n-1}$
with  $\eta_1 = r e^{i \theta}$ and $0\leq r \leq 1$, we have
\begin{align*}
\mathbf{k}_{\mathcal{H}_{p,q}(\mathbb{S}^{n-1})}(e_1, \eta)
= N_{n,p,q}\,r^{|p-q|}  e^{i(p-q)\theta}
\,
\frac{(p\wedge q)!(n-2)!}{((p\wedge q) + (n-2))!}
\,
P_{p\wedge q}^{n-2,|p-q|}(2r^2 -1)\,.
\end{align*}
Then by  the integral representation of  the projection constant of $\mathcal{H}_{p,q}(\mathbb{S}^{n-1})$ given in
Theorem~\ref{top}, Equation~\eqref{integration} and the substitution  $t=r^2 $, we get
\begin{align*}
\boldsymbol{\lambda}\big(\mathcal{H}_{p,q}(\mathbb{S}^{n-1})\big)
&=
N_{n,p,q}\,
\frac{(p\wedge q)!(n-2)!}{((p\wedge q) + (n-2))!}
\,
\frac{n-1}{\pi} 2 \pi  \int_{0}^{1} r^{|p-q|+1}  (1-r^2)^{n-2}
\big| P_{p\wedge q}^{n-2,|p-q|}(2r^2 -1) \big| \,dr
\\&
=
N_{n,p,q}\,
\frac{(p\wedge q)!(n-2)!}{((p\wedge q) + (n-2))!}
\,
(n-1)  \int_{0}^{1}   (1-t)^{n-2} t^{\frac{|p-q|}{2}}
\big| P_{p\wedge q}^{n-2,|p-q|}(2t -1) \big| \,dt\,.
\end{align*}
Finally,  the  equality first claimed follows from the dimension formula from Equation~\eqref{dimfor-comp},
and the second one clearly is a consequence of the substitution  $s = 2t-1 $.
\end{proof}

We now examine the general bihomogeneous case.

\begin{theorem} \label{bi-end}
Let $n \geq 2 $ and $p,q \geq 0 $. Then
 \begin{equation*}
  \boldsymbol{\lambda}\big(\mathfrak{P}_{p,q}(\mathbb{S}^{n-1})\big)
  \,\,= \,\, \frac{\Gamma(n + p \vee q )}{\Gamma(n-1)\Gamma( 1 + p \vee q )}\,\, \int_{0}^{1} (1-t)^{n-2} t^{\frac{|p-q|}{2}} \,\big| P_{p\wedge q}^{n-1,|p-q|}(2t -1)\big| dt\,,
\end{equation*}
or equivalently,
\begin{equation*}
  \boldsymbol{\lambda}\big(\mathfrak{P}_{p,q}(\mathbb{S}^{n-1})\big)
 \,\, = \frac{1}{2^{n + \frac{|p-q|}{2}}}\,\,  \frac{\Gamma(n + p \vee q )}{\Gamma(n-1)\Gamma( 1 + p \vee q )}\,\, \int_{-1}^{1} (1-s)^{n-2} (1+s)^{\frac{|p-q|}{2}} \,\big| P_{p\wedge q}^{n-1,|p-q|}(s)\big| ds\,.
\end{equation*}
  \end{theorem}

Before the proof, we present an observation which will be crucial many times, so we isolate it as a remark.

\begin{remark} \label{iso piola}
Let $ p, q \geq 0 $. The Banach spaces $ \mathcal{H}_{p,q}(\mathbb{S}^{n-1}) $ and $ \mathcal{H}_{q,p}(\mathbb{S}^{n-1}) $ are isometrically
isomorphic. Similarly, the Banach spaces $ \mathfrak{P}_{p,q}(\mathbb{S}^{n-1}) $ and $ \mathfrak{P}_{q,p}(\mathbb{S}^{n-1}) $ are also
isometrically isomorphic. In particular, these spaces share the same projection constant. Of course, the isomorphism of the previous remark is established through the mapping $ f \mapsto \tilde{f} $, where
$ \tilde{f}(z) := f(\overline{z})$.
\end{remark}

\begin{proof}[Proof of Theorem~\ref{bi-end}]
In a first step we show that for every  $\eta \in \mathbb{S}^{n-1} $ with  $\eta_1 = r e^{i \theta}  $
\begin{equation}  \label{start}
  \mathbf{k}_{\mathfrak{P}_{p,q}(\mathbb{S}^{n-1})}(e_1, \eta)
    \, =   \,
    r^{|p-q|}  e^{i(p-q)\theta}  \, \frac{ \Gamma(  p \wedge q + n)}{\Gamma(n)\Gamma(1+p \wedge q)}
P_{p\wedge q}^{n-1, q-p}(2r^2 -1)\,,
 \end{equation}
and according to Lemma~\ref{accessibility-comp} we start to modify
\[
 \mathbf{k}_{\mathfrak{P}_{p,q}(\mathbb{S}^{n-1})}(e_1, \eta) =   \sum_{j=0}^{p\wedge q} \,N^\mathbb{C}_{n,p-j,q-j}\,L^{\diamond}_{n,p-j, q-j}(\eta_1)\,.
\]
By Theorem~\ref{Koornwinder} we have that for every  $\eta \in \mathbb{S}^{n-1} $ with  $\eta_1 = r e^{i \theta}  $
\[
 \mathbf{k}_{\mathfrak{P}_{p,q}(\mathbb{S}^{n-1})}(e_1, \eta) \, =   \,
  r^{|p-q|}  e^{i(p-q)\theta} \sum_{j=0}^{p\wedge q} \,
    \frac{((p-j)\wedge (q-j))!(n-2)!}{(((p-j)\wedge (q-j)) + (n-2))!}
   \,
  N^\mathbb{C}_{n,p-j,q-j}
  \,P_{(p-j)\wedge (q-j)}^{n-2,|p-q|}(2r^2 -1)\,.
\]
The aim now is to apply the following general addition formula of Szeg\"o's classical book \cite[(4.5.3), p.~71]{szeg1939orthogonal}: for every appropriate choice of  $\alpha, \beta, d $ and  $x $
\begin{equation}\label{SZ}
  \sum_{\nu=0}^d  \frac{(2\nu+\alpha+ \beta +1) \Gamma(\nu+\alpha+ \beta +1)}
{\Gamma(\nu+ \beta +1)} P_{\nu}^{\alpha, \beta}(x)
=
\frac{ \Gamma( d+\alpha+ \beta +2)}{\Gamma(d+\beta+1)}
P_{d}^{\alpha+1, \beta}(x)\,.
\end{equation}
In order to use this formula, we assume first that  $p\wedge q = p $, and put  $\alpha = n-2 $ and  $\beta = |p-q|= q-p $.
Then, using Equation~\eqref{dimfor-comp}, we get
\begin{align*}
&
\sum_{j=0}^{p} \,
  \frac{(p-j)! (n-2)!}{((p-j)+(n-2))!}
 \,
  N^\mathbb{C}_{n,p-j,q-j}
  \,P_{p-j}^{n-2,q-p}(2r^2 -1)
\\&
  \, =   \,
  \sum_{j=0}^{p} \,
  \frac{(p-j)! (n-2)!}{((p-j)+(n-2))!}
 \,
    \frac{(n+(p-j)+(q-j)-1) (n-2+(p-j))! (n-2+(q-j))!}{(p-j)!(q-j)!(n-1)!(n-2)!}
    \,P_{p-j}^{n-2,q-p}(2r^2 -1)
    \\&
  \, =   \,
  \sum_{j=0}^{p} \,
     \frac{(n+(p-j)+(q-j)-1)  (n-2+(q-j))!}{(q-j)!(n-1)!}
      \,P_{p-j}^{n-2,q-p}(2r^2 -1)
    \,,
  \end{align*}
  and consequently the substitution  $k = p-j $ leads to
  \begin{align*}
\sum_{j=0}^{p} \,
  \frac{(p-j)! (n-2)!}{((p-j)+(n-2))!}
 \,&
  N^\mathbb{C}_{n,p-j,q-j}
  \,P_{p}^{n-2,q-p}(2r^2 -1)
\\&
  \, =   \,
    \sum_{k=0}^{p} \,
     \frac{(n+2k +q-p-1)  (n+k +q-p-2)!}{(q+k-p)!(n-1)!}
      \,P_{k}^{n-2,q-p}(2r^2 -1)
    \,.
  \end{align*}
  On the other hand, by Equation~\eqref{SZ} for every  $0\leq r \leq 1  $
  \begin{align*}
        \sum_{k=0}^p & \frac{(2k+(n-2)+ (q-p) +1) \Gamma(k+(n-2)+ (q-p) +1)}
{\Gamma(k+ (q-p) +1)} P_{k}^{n-2, q-p}(2r^2 -1)
\\&
=
\frac{ \Gamma( p+(n-2)+ (q-p) +2)}{\Gamma(p+(q-p)+1)}
P_{p}^{(n-2)+1, q-p}(2r^2 -1)\,.
  \end{align*}
  Comparing the coefficients of the preceding two sums, we see that
  \[
  \frac{(n+2k +q-p-1)  (n+k +q-p-2))!}{(q+k-p)!(n-1)!} \,=\,\frac{(2k+(n-2)+ (q-p) +1) \Gamma(k+(n-2)+ (q-p) +1)}
{\Gamma(k+ (q-p) +1)(n-1)!}\,,
  \]
  so that all together, for every  $\eta \in \mathbb{S}^{n-1} $ with  $\eta_1 = r e^{i \theta}  $, we get
  \begin{align*}
        \mathbf{k}_{\mathfrak{P}_{p,q}(\mathbb{S}^{n-1})}(e_1, \eta) &\, =   \,
    \frac{1}{(n-1)!}\,
      r^{q-p}  e^{i(p-q)\theta} \sum_{j=0}^{p} \,
  \frac{(p-j)! (n-2)!}{(p-j)+(n-2)!}
 \,
  N^\mathbb{C}_{n,p-j,q-j}
  \,P_{p-j}^{n-2,q-p}(2r^2 -1)
  \\&
  \, =   \,
  \frac{1}{(n-1)!}\,
   r^{q-p}  e^{i(p-q)\theta} \,
\frac{ \Gamma( p+(n-2)+ (q-p) +2)}{\Gamma(p+(q-p)+1)}
P_{p}^{(n-2)+1, q-p}(2r^2 -1)
\\&
 \, =   \,
 \frac{\Gamma(n + q)}{\Gamma(n)\Gamma( 1 + q )}\,
   r^{q-p}  e^{i(p-q)\theta} \,P_{p}^{n-1, q-p}(2r^2 -1)\,.
  \end{align*}
  For the case that  $p\wedge q = p $ this is the desired result from Equation~\eqref{start}, and consequently for the dual case
  $p\wedge q = q $ this follows from the fact that the spaces  $\mathfrak{P}_{p,q}(\mathbb{S}^{n-1}) $ and  $\mathfrak{P}_{q,p}(\mathbb{S}^{n-1}) $
  are isomorphic as Banach spaces (see Remark~\ref{iso piola}).
  \end{proof}

\bigskip

\section{Asymptotic Behavior and Various Bounds}

\subsection{Equidistant parameters}

We now investigate the appropriate asymptotic behavior in the scenario where the parameters \( (p,q) \) take the form \( (p, p+d) \) for a fixed distance \( d \) as \( p \) grows to infinity (i.e., \( p \) and \( q \) remain equidistant). Note that \( \mathcal{H}_{p,p+d}(\mathbb{S}^{n-1}) \) is well-defined only if \( p + d \geq 0 \).

\begin{proposition} \label{asymp-harm}
For \( d \in \mathbb{Z} \) and \( n \ge 2 \),
\[\lim_{p \to + \infty}  \frac{\boldsymbol{\lambda}\big(\mathcal{H}_{p,p+d}(\mathbb{S}^{n-1})\big)}{p^{n- \frac{3}{2}}}\,\,=\,\,\frac{2\Gamma\left(\frac{2n-1}{4} \right) \Gamma\left(\frac{3}{4}\right) }{\pi^{\frac{3}{2}}\Gamma(n-1)\Gamma\left(\frac{n + 1}{2}\right)}\,.\]
\end{proposition}

Our proof uses the  following well-known formula for the asymptotic of ratios of Gamma functions:
\begin{align}\label{mainasym}
\lim_{x\to + \infty} \frac{\Gamma(x + a)}{ \Gamma(x + b) x^{a-b}} = 1, \quad\,   a, b >0\,.
\end{align}
This formula combined with  Equation~\eqref{dimfor-comp} in particular proves  that for fixed  $n $ and large  $p $
\[
\dim \big(\mathcal{H}_{p,p+d}(\mathbb{S}^{n-1})\big) \sim_{c(n)} p^{2n -3}\,,
\]
showing that Proposition~\ref{asymp-harm} meets the Kadets-Snobar theorem. Recall that this result asserts that the projection constant
of any finite dimensional Banach space $X$ is bounded above by  $ \sqrt{\dim X}  $  (see, for instance, \cite[Theorem 10, III.B.]{wojtaszczyk1996banach}).

We will also make use of a simple but useful remark, which allows us to express a certain integral in terms of the Beta function. This
result will play a key role in simplifying subsequent computations.

\begin{remark}\label{rmk: function c}
The function $\boldsymbol{C}(u,v)$, defined as
\begin{equation}\label{eq: definition of C}
\boldsymbol{C}(u,v): = \int_{0}^{\frac{\pi}{2}} \left( \sin \frac{\theta}{2} \right)^{u} \left( \cos \frac{\theta}{2} \right)^{v}d\theta
+
\int_{0}^{\frac{\pi}{2}}  \left( \cos \frac{\theta}{2} \right)^{u} \left( \sin \frac{\theta}{2} \right)^{v}d\theta\,,
\end{equation}
where $u,v > 0$, can be expressed in terms of the Beta function as
\[
\boldsymbol{C}(u,v) = \boldsymbol{B}\left(\frac{u+1}{2}, \frac{v+1}{2}\right) = \frac{\Gamma\left(\frac{u+1}{2}\right) \Gamma\left(\frac{v+1}{2}\right)}{\Gamma\left(\frac{u+v+2}{2}\right)}\,.
\]
\end{remark}

\begin{proof}
Using the substitution $t = \frac{\theta}{2}$, we rewrite $\boldsymbol{C}(u,v)$ as
\begin{equation} \label{eq: def C}
2 \int_{0}^{\frac{\pi}{4}} \left( \sin t \right)^u \left( \cos t \right)^v dt + 2 \int_{0}^{\frac{\pi}{4}} \left( \cos t \right)^u \left( \sin t \right)^v dt.
\end{equation}
Next, we focus on the second integral and apply the substitution $s = \frac{\pi}{2} - t$. Using the trigonometric identities $\cos\left(\frac{\pi}{2} - s\right) = \sin s$ and $\sin\left(\frac{\pi}{2} - s\right) = \cos s$, we express this integral as
\[
\int_{0}^{\frac{\pi}{4}} \left( \cos t \right)^u \left( \sin t \right)^v dt = \int_{\frac{\pi}{4}}^{\frac{\pi}{2}} \left( \sin s \right)^u \left( \cos s \right)^v ds.
\]
Thus, if we replace this term in Equation \eqref{eq: def C}, we have
\[
\boldsymbol{C}(u,v) = 2 \int_{0}^{\frac{\pi}{2}} \left( \sin t \right)^u \left( \cos t \right)^v dt.
\]
Now, we use the known relation between the Beta function $B(x,y)$ and trigonometric integrals: $
B(x,y)  \linebreak = 2 \int_{0}^{\frac{\pi}{2}} \left( \sin t \right)^{2x-1} \left( \cos t \right)^{2y-1} dt.$
In our case, comparing the integrands, we see that$$
\boldsymbol{C}(u,v) = B\left(\frac{u+1}{2}, \frac{v+1}{2}\right).$$
The proof is concluded if we recall the definition of the Beta function $B(x, y)$ in terms of Gamma functions as $
B(x, y) = \frac{\Gamma(x) \Gamma(y)}{\Gamma(x + y)}.$
\end{proof}

We are ready to prove Proposition \ref{asymp-harm}.

\begin{proof}[Proof of Proposition~\ref{asymp-harm}]
By Remark \ref{iso piola}, it suffices to prove the case $d\ge 0.$
By Theorem~\ref{jacobi} we have
\begin{align*}
&
\frac{\boldsymbol{\lambda}\big(\mathcal{H}_{p,p+d}(\mathbb{S}^{n-1})\big)}{p^{n-\frac{3}{2}}}
\\&
=
\frac{1}{\Gamma(n-1) 2^{n + \frac{d}{2}}}
\,
\frac{n+2p +d-1}{p }
\,
\frac{\Gamma(n + p + d - 1)}{ p^{n-2} \Gamma(p+d +1)  }
\,\,
\bigg(
  \sqrt{p} \int_{-1}^{1} (1-s)^{n-2} (1+s)^{\frac{d}{2}} \,\big| P_{p}^{n-2,d}(s)\big|\,ds
 \bigg)\,.
  \end{align*}
Clearly, if  $p $ tends to infinity, then the second factor converges to  $2 $ and  the third factor by Equation~\eqref{mainasym}
to  $1 $, and hence it remains to handle the last factor.
Since  $2(n-2) - (n-2) + \frac{3}{2} > 0 $, it follows from   \cite[§19, Equation (6\raisebox{-0.5ex}{4}\raisebox{-1ex}{1}), p. 84]{szeg1933Asymptotische}  that
\[
\lim_{p\to + \infty} \sqrt{p} \int_{0}^{1} (1-s)^{n-2} (1+s)^{\frac{d}{2}} \,\big| P_{p}^{n-2,d}(s)\big| ds =
\frac{2^{n+\frac{d}{2}}}{\pi^{\frac{3}{2}}}
\int_{0}^{\frac{\pi}{2}} \big(\sin \frac{\theta}{2}\big)^{n-\frac{3}{2}} \big(\cos \frac{\theta}{2}\big)^{\frac{1}{2}}d\theta\,.
\]
 On the other hand, since by
        \cite[ 4.1.3, p.~59]{szeg1939orthogonal}
       \begin{align*} \label{SZ1}
      P_p^{n-2, d} (t) = (-1)^d P_p^{d,n-2} (-t), \quad\, t\in [-1, 1]\,,
    \end{align*}
we have
\begin{align*}
  \int_{-1}^{0} (1-s)^{n-2} (1+s)^{\frac{d}{2}} \,\big| P_{p}^{n-2,d}(s)\big|\,ds
  =
  \int_{0}^{1} (1-u)^{\frac{d}{2}}(1+u)^{n-2}  \,\big| P_{p}^{d,n-2}(u)\big|\,du\,,
\end{align*}
and hence again by Szeg\"o's result from \cite[§19, Equation (6\raisebox{-0.5ex}{4}\raisebox{-1ex}{1}), p. 84]{szeg1933Asymptotische}
\[
\lim_{p\to + \infty} \sqrt{p} \int_{-1}^{0} (1-s)^{n-2} (1+s)^{\frac{d}{2}} \,\big| P_{p}^{n-2,d}(s)\big| ds =
\frac{2^{n+\frac{d}{2}}}{\pi^{\frac{3}{2}}}
\int_{0}^{\frac{\pi}{2}} \big(\sin \frac{\theta}{2}\big)^{\frac{1}{2}}\big(\cos \frac{\theta}{2}\big)^{n-\frac{3}{2}} d\theta\,,
\]
where we now use that  $2 \frac{d}{2} - d + \frac{3}{2} > 0 $. Combining everything, we get that
\[
\lim_{p \to + \infty}  \frac{\boldsymbol{\lambda}\big(\mathcal{H}_{p,p+d}(\mathbb{S}^{n-1})\big)}{p^{n- \frac{3}{2}}}
= \frac{2}{\Gamma(n-1) 2^{n + \frac{d}{2}}}  \frac{2^{n+\frac{d}{2}}}{\pi^{\frac{3}{2}}} \boldsymbol{C}\Big(n-\frac{3}{2},\frac{1}{2}\Big)
=
\frac{2 }{\pi^{\frac{3}{2}}\Gamma(n-1)} \boldsymbol{C}\Big(n-\frac{3}{2},\frac{1}{2}\Big)\,,
\]
where $\boldsymbol{C}$ is the function defined in Equation \eqref{eq: definition of C}.
To conclude the proof we invoke Remark \ref{rmk: function c} to see that
\[
\boldsymbol{C}\left(n - \frac{3}{2}, \frac{1}{2}\right) =  \frac{\Gamma\left(\frac{2n-1}{4}\right)
\Gamma\left(\frac{3}{4}\right)}{\Gamma\left(\frac{n + 1}{2}\right)}\,. \qedhere
\]
\end{proof}

\smallskip

We now turn our attention to the asymptotic behavior of the $d$-stationary case (i.e., where the difference of the indices is exactly $d$). Thus, we study \( \boldsymbol{\lambda}\big(\mathfrak{P}_{p,p+d}(\mathbb{S}^{n-1})\big) \) when the degree \( p \) increases to infinity. Note that \( \mathfrak{P}_{p,p+d}(\mathbb{S}^{n-1}) \) is well-defined only if \( p + d \geq 0 \).

\begin{proposition} \label{end}
For \( d \in \mathbb{Z} \) and \( n \ge 2 \),
\[\lim_{p \to + \infty}  \frac{\boldsymbol{\lambda}\big(\mathfrak{P}_{p,p+d}(\mathbb{S}^{n-1})\big)}{p^{n- \frac{3}{2}}}  =\frac{\Gamma\left(\frac{2n - 3}{4}\right)\Gamma \left(\frac{3}{4}\right)}{\pi^{3/2}\Gamma(n-1)  \Gamma\left(\frac{n}{2}\right)}\,.\]
\end{proposition}

  \smallskip
 \noindent  Note that the preceding limit is independent of  $d $, which at first glance seems surprising.
Moreover, the result here shows that the asymptotic increase of  $\boldsymbol{\lambda}\big(\mathfrak{P}_{p,p+d}(\mathbb{S}^{n-1})\big) $ for fixed  $n $ and large  $p $ doesn't meet the Kadets-Snobar theorem. Indeed, we have that
\[
\lim_{p \to + \infty}  \frac{\Gamma(n)}{p^{n-1}}\,\dim \mathfrak{P}_{p,0}(\mathbb{S}^{n-1})
=
\lim_{p \to + \infty}
 \frac{\Gamma(n)}{p^{n-1}}\,
\frac{\Gamma(n+p)}{\Gamma(p+1) \Gamma(n)}
=1\,,
\]
and hence
\[
\lim_{p,q \to + \infty}  \frac{\dim \mathfrak{P}_{p,q}(\mathbb{S}^{n-1})}{p^{n-1} q^{n-1}}
=
\lim_{p \to + \infty}  \frac{\dim \mathfrak{P}_{p,0}(\mathbb{S}^{n-1})}{p^{n-1} }\,\,
\lim_{q \to + \infty}  \frac{\dim \mathfrak{P}_{q,0}(\mathbb{S}^{n-1})}{q^{n-1} }
=
\frac{1}{\Gamma(n)^2}\,,
\]
and in particular
\[
\lim_{p \to + \infty}  \frac{\sqrt{\dim \mathfrak{P}_{p,p+d}(\mathbb{S}^{n-1})}}{p^{n-1} }=
\frac{1}{\Gamma(n)}\,.
\]

\begin{proof}[Proof of Proposition~\ref{end}] We only sketch the argument which is basically exactly the same as that for  Corollary~\ref{asymp-harm}. By Remark \ref{iso piola} we need only to check the case $d\ge 0.$ From Theorem~\ref{bi-end} we have
\begin{align*}
\frac{\boldsymbol{\lambda}\big(\mathfrak{P}_{p,p+d}(\mathbb{S}^{n-1})\big)}{p^{n-\frac{3}{2}}}
=
\frac{1}{\Gamma(n-1) 2^{n + \frac{d}{2}}}
\,
\frac{\Gamma(n + p + d)}{ p^{n-1} \Gamma(p+d +1)  }
\,\,
\bigg(
  \sqrt{p} \int_{-1}^{1} (1-s)^{n-2} (1+s)^{\frac{d}{2}} \,\big| P_{p}^{n-1,d}(s)\big|\,ds
 \bigg)\,.
  \end{align*}
  Clearly, if  $p $ tends to infinity, then the second factor converges to  $1 $, and so we again only have to handle  the last integral factor. By  \cite[§19, Equation (6\raisebox{-0.5ex}{4}\raisebox{-1ex}{1}), p. 84]{szeg1933Asymptotische}
\[
\lim_{p\to + \infty} \sqrt{p} \int_{0}^{1} (1-s)^{n-2} (1+s)^{\frac{d}{2}} \,\big| P_{p}^{n-1,d}(s)\big|\,ds =
\frac{2^{n+\frac{d}{2}}}{\pi^{\frac{3}{2}}}
\int_{0}^{\frac{\pi}{2}} \big(\sin \frac{\theta}{2}\big)^{n-\frac{5}{2}} \big(\cos \frac{\theta}{2}\big)^{\frac{1}{2}}d\theta\,,
\]
since  $2(n-2) - (n-1) + \frac{3}{2} = n- \frac{3}{2} >0  $. On the other hand (as in the proof of Proposition~\ref{asymp-harm}),
\[
\lim_{p\to + \infty} \sqrt{p} \int_{-1}^{0} (1-s)^{n-2} (1+s)^{\frac{d}{2}} \,\big| P_{p}^{n-1,d}(s)\big|\,ds =
\frac{2^{n+\frac{d}{2}}}{\pi^{\frac{3}{2}}}
\int_{0}^{\frac{\pi}{2}} \big(\sin \frac{\theta}{2}\big)^{\frac{1}{2}}\big(\cos \frac{\theta}{2}\big)^{n-\frac{5}{2}}\,d\theta\,.
\]
Together, we obtain that
\[
\lim_{p \to + \infty}  \frac{\boldsymbol{\lambda}\big(\mathfrak{P}_{p,p+d}(\mathbb{S}^{n-1})\big)}{p^{n- \frac{3}{2}}}
= \frac{1}{\Gamma(n-1) 2^{n + \frac{d}{2}}}  \frac{2^{n+\frac{d}{2}}}{\pi^{\frac{3}{2}}} \boldsymbol{C}\Big(n-\frac{5}{2},\frac{1}{2}\Big)
=
\frac{1}{\pi^{\frac{3}{2}}\Gamma(n-1)} \boldsymbol{C}\Big(n-\frac{5}{2},\frac{1}{2}\Big)\,,
\]
where $\boldsymbol{C}$ is the function defined in Equation \eqref{eq: definition of C}.
Once again we use Remark \ref{rmk: function c} to see that
\[
\boldsymbol{C}\left(n - \frac{5}{2}, \frac{1}{2}\right) = \frac{\Gamma\left(\frac{2n - 3}{4}\right)  \Gamma\left(\frac{3}{4}\right)}{ \Gamma\left(\frac{n}{2}\right)}. \qedhere
\]
\end{proof}

\smallskip

\subsection{Upper estimates}

At this point, we turn our attention to providing upper bounds for the projection constant of the spaces under consideration.

\begin{proposition} \label{prop: bound bihom}
  Let   $n \ge 2  $ and $p,q \in \mathbb{N}_0$. Then,  $\boldsymbol{\lambda}\big(\mathfrak{P}_{p,q}(\mathbb{S}^{n-1})\big) $ is bounded by
  \begin{equation}\label{uno}
\frac{1}{\Gamma(n-1)}  \frac{\Gamma(n+p \wedge q)}{\Gamma(1+p) \Gamma(1+q)}
\sum_{m=0}^{p \wedge q} \binom{p \wedge q}{m} \frac{\Gamma(n+m+p \vee q)}{\Gamma(n+m)}
\int_{0}^{1} (1+t)^{n+m-2} t^{\frac{\vert p - q \vert}{2}}  dt\,,
 \end{equation}
So, if  $p \geq q $,
      \begin{equation}\label{due}
      \boldsymbol{\lambda}\big(\mathfrak{P}_{p,q}(\mathbb{S}^{n-1})\big)
  \leq
  \frac{1}{\Gamma(n-1)}  \frac{\Gamma(n+q)}{\Gamma(1+q)}
  \sum_{m=0}^{q} \binom{q}{m}
    \frac{\Gamma(n+m-1)}{\Gamma(n+m)}\frac{\Gamma(n+m+p)}{\Gamma(1+p)}
    \frac{\Gamma(\frac{p-q}{2}+1)}{\Gamma(n+m+\frac{p-q}{2})}\,.
      \end{equation}
    \vspace{2mm}
    In particular,
       \begin{equation}\label{tres}
       \limsup_{p \to + \infty} \, \boldsymbol{\lambda}\big(\mathfrak{P}_{p,q}(\mathbb{S}^{n-1})\big)
  \leq \frac{2^{n-1}}{\Gamma(n-1)}  \frac{\Gamma(n+q)}{\Gamma(1+q)}
  \sum_{m=0}^{q} \binom{q}{m}
    \frac{2^{m}}{n+m-1} \le 2^{n-1}3^q\binom{n-1+q}{q}
  \,.
      \end{equation}
    \end{proposition}

  Note that this result in the special case where  $q=0 $ recovers the bound   $\boldsymbol{\lambda}\big(\mathcal{P}_{p}(\ell_2^n(\mathbb{C}))\big)
  \leq 2^{n-1} $ proved by  Ryll and Wojtaszczyk in \cite[Proposition 1.1]{ryll1983homogeneous}.

\begin{proof}
Suppose first that  $p \geq q $, the other case follows from Remark~\ref{iso piola}. We use the following well-known explicit formula for Jacobi polynomials from \cite[Equation (4.21.2)]{szeg1939orthogonal}:
    \[
  P_d^{\alpha, \beta} (t)
  \,= \,
    \frac{\Gamma(\alpha + d +1)}{d!\Gamma(\alpha + \beta +d+1)}
          \sum_{m=0}^{d} \binom{d}{m}
    \frac{\Gamma(\alpha + \beta + d + m +1)}{\Gamma(\alpha+m +1)}
    \bigg( \frac{t-1}{2}\bigg)^m\,,
      \]
      and so in our special case
        \[
  P_q^{n-1, p-q} (t)
   \,= \,
    \frac{\Gamma(n+q)}{\Gamma(q+1)\Gamma(n+p)}
                 \sum_{m=0}^{q} \binom{q}{m}
                     \frac{\Gamma(n+m+p)}{\Gamma(n+m)}
    \bigg( \frac{t-1}{2}\bigg)^m\,.
      \]
      Then by Theorem~\ref{bi-end} and the triangle inequality
    \begin{align*}
  \boldsymbol{\lambda}\big(\mathfrak{P}_{p,q}(\mathbb{S}^{n-1})\big)
  &
    \,= \, \frac{\Gamma(n + p )}{\Gamma(n-1)\Gamma( 1 + p  )}
    \, \int_{0}^{1} (1-t)^{n-2} t^{\frac{p-q}{2}} \,\big| P_{ q}^{n-1,p-q}(2t -1)\big| dt
    \\&
    \,\leq \, \frac{\Gamma(n +q )}{\Gamma(n-1)\Gamma( 1 + p  )\Gamma( 1 + q  )}
    \,
    \sum_{m=0}^{q} \binom{q}{m}
                     \frac{\Gamma(n+m+p)}{\Gamma(n+m)}
                      \int_{0}^{1} (1-t)^{n-2} t^{\frac{p-q}{2}}
    \bigg| \frac{2t-2}{2}\bigg|^m dt
    \\&
    \,\leq \, \frac{\Gamma(n +q )}{\Gamma(n-1)\Gamma( 1 + p  )\Gamma( 1 + q  )}
    \,
    \sum_{m=0}^{q} \binom{q}{m}
                     \frac{\Gamma(n+m+p)}{\Gamma(n+m)}
                      \int_{0}^{1} (1-t)^{n+m-2} t^{\frac{p-q}{2}}
     dt\,.
\end{align*}
This is Estimate~\eqref{uno}, and then also Estimate~\eqref{due} follows from the standard reformulation of Beta functions in terms of Gamma functions.
To see the Estimate~\eqref{tres}, take the limit in Estimate~\eqref{due} as  $p $ tends infinity and
note  that by Equation~\eqref{mainasym}
\begin{align*}
\lim_{p \to + \infty} \frac{\Gamma(n+m+p)}{\Gamma(1+p)}
&
    \frac{\Gamma(\frac{p}{2}-\frac{q}{2}+1)}{\Gamma(n+m+\frac{p}{2}-\frac{q}{2})}
    \\&
    \,=\,
\lim_{p \to + \infty} \frac{\Gamma(n+m+p)}{\Gamma(1+p) p^{n+m-1}}
    \frac{\Gamma(\frac{p}{2}-\frac{q}{2}+1)}{\Gamma(n+m+\frac{p}{2}-\frac{q}{2})  p^{1-m-n}
        }
        \\&
    \,=\,
\lim_{p \to + \infty} \frac{\Gamma(n+m+p)}{\Gamma(1+p) p^{n+m-1}}
    \frac{\Gamma(\frac{p}{2}-\frac{q}{2}+1) 2^{m+n-1}}{\Gamma(n+m+\frac{p}{2}-\frac{q}{2}) \Big(\frac{p}{2}\Big)^{1-m-n}}\,=\,  2^{m+n-1}\,.
    \qedhere
  \end{align*}
\end{proof}

\smallskip

 Using Theorem~\ref{jacobi} instead  of Theorem~\ref{bi-end}, similar considerations lead to  the following harmonic counterpart.

\begin{proposition} \label{prop: bound harm}
  Let   $n \ge 2  $ and $p,q \in \mathbb{N}_0$. Then  $\boldsymbol{\lambda}\big(\mathcal{H}_{p,q}(\mathbb{S}^{n-1})\big) $ is bounded by
  \begin{equation}\label{uno-H}
  \frac{(n+p+q-1)\Gamma(n-1+p \wedge q)}{\Gamma(n-1)\Gamma(1+p) \Gamma(1+q)}
  \sum_{m=0}^{p \wedge q} \binom{p \wedge q}{m} \frac{\Gamma(n-1+m+p \vee q)}{\Gamma(n-1+m)}
  \int_{0}^{1} (1+t)^{n+m-2} t^{\frac{\vert p-q \vert}{2}}  dt\,,
  \end{equation}
  So, if  $p\geq q $,
      \begin{equation}\label{due-H}
      \boldsymbol{\lambda}\big(\mathcal{H}_{p,q}(\mathbb{S}^{n-1})\big)
      \leq
  \frac{\Gamma(n-1+q)}{\Gamma(n-1) \Gamma(1+q)}
  \sum_{m=0}^{q} \binom{q}{m}(n+p+q-1)
    \frac{\Gamma(n-1+m+p)}{\Gamma(1+p)}
        \frac{\Gamma(\frac{p-q}{2}+1)}{\Gamma(n+m+\frac{p-q}{2})}\,.
      \end{equation}
    \vspace{2mm}
    In particular,
       \begin{equation}\label{tres-H}
       \limsup_{p \to + \infty} \, \boldsymbol{\lambda}\big(\mathcal{H}_{p,q}(\mathbb{S}^{n-1})\big)
  \leq \frac{2^{n-1}}{\Gamma(n-1)}  \frac{\Gamma(n-1+q)}{\Gamma(1+q)}
  \sum_{m=0}^{q} \binom{q}{m} 2^{m} = 2^{n-1}3^q\binom{n-2+q}{q}
      \,.
      \end{equation}
    \end{proposition}

We note that as a consequence both  sequences of projection constants $\left(\boldsymbol{\lambda}\big(\mathcal{H}_{p,q}(\mathbb{S}^{n-1})\big)\right)_{p\in\mathbb N}$ and $\left(\boldsymbol{\lambda}\big(\mathfrak{P}_{p,q}(\mathbb{S}^{n-1})\big)\right)_{p\in\mathbb N}$ are bounded, although
their dimensions increase to infinity when  $p $ does:
  \[
  \text{
   $\dim \mathcal{H}_{p,q}(\mathbb{S}^{n-1}) = \binom{n-2+p}{p}\binom{n-2+q}{q}\frac{n-1+p+q}{n-1} $
  \quad and \quad
   $\dim \mathfrak{P}_{p,q}(\mathbb{S}^{n-1}) = \binom{n-1+p}{p}\binom{n-1+q}{q} $\,.}
  \]

\subsection{The case  $\pmb{q=1} $}
We now consider the case where the parameter  $ q $ equals  $1  $. First we note that a careful look to the calculations of the previous section give the following estimates.
\begin{proposition}\label{prop: estimate q=2}
For every $n\ge 2$ and and $p \in \mathbb{N}_0$,
       \begin{equation*}
(n-1)2^{n-1} \le       \liminf_{p \to + \infty} \, \boldsymbol{\lambda}\big(\mathcal{H}_{p,1}(\mathbb{S}^{n-1})\big)
  \leq \limsup_{p \to + \infty} \, \boldsymbol{\lambda}\big(\mathcal{H}_{p,1}(\mathbb{S}^{n-1})\big)
  \leq 3(n-1){2^{n-1}},
  \end{equation*}
and 
\begin{equation*}
(n-2)2^{n-1} \le       \liminf_{p \to + \infty} \, \boldsymbol{\lambda}\big(\mathfrak{P}_{p,1}(\mathbb{S}^{n-1})\big)
  \leq \limsup_{p \to + \infty} \, \boldsymbol{\lambda}\big(\mathfrak{P}_{p,1}(\mathbb{S}^{n-1})\big)
  \leq (3n-2){2^{n-1}}.
  \end{equation*}
\end{proposition}
The upper bounds are those in \eqref{tres-H}
and \eqref{tres}. The lower bounds follow a similar approach: since the integral formula contains only two terms, the projection constant can be bounded below by the difference of their values.

Now we deal with the case  $p \geq 2 $ in the two-dimensional setting.

\begin{proposition} \label{nazibude}
  For every  $p \ge 2  $
  \[
  \boldsymbol{\lambda}\big(\mathcal{H}_{p,1}(\mathbb{S}^{1})\big)
  \,= \,(p+2)  \left( \frac{8}{p+3}\Big( \frac{p}{p+1}\Big)^{\frac{p+3}{2}} + \frac{2}{p+1} -\frac{4}{p+3}\right)
  \]
  and
  \[
  \boldsymbol{\lambda}\big(\mathfrak{P}_{p,1}(\mathbb{S}^{1})\big)
  \,= \,(p+1)  \left( \frac{8(p+2)}{(p+1)(p+3)}\Big( \frac{p}{p+2}\Big)^{\frac{p+3}{2}} + \frac{4}{p+1} -\frac{4(p+2)}{(p+1)(p+3)}\right)\,.
  \]
  In particular,
  \[
  \lim_{p \to + \infty} \boldsymbol{\lambda}\big(\mathcal{H}_{p,1}(\mathbb{S}^{1})\big)
    =
  \frac{8}{\sqrt{e}} -2
  \]
  and
    \[
    \lim_{p \to + \infty} \boldsymbol{\lambda}\big(\mathfrak{P}_{p,1}(\mathbb{S}^{1})\big)
  =
  \frac{8}{e}\,.
  \]
      \end{proposition}

      \smallskip

% We note that as a consequence both  sequences of projection constants are bounded as sequences of  $p$, although
% their dimensions increase to infinity when  $p $ does:
%   \[
%   \text{
%    $\dim \mathcal{H}_{p,1}(\mathbb{S}^{1}) = p+2 $
%   \quad and \quad
%    $\dim \mathfrak{P}_{p,1}(\mathbb{S}^{1}) =  2(p+1) $\,.}
%   \]

      \begin{proof}[Proof of Proposition~\ref{nazibude}]
             The proof is simple - note first that
 \begin{equation}\label{degree1}
   P_1^{\alpha, \beta}(2t -1) = \frac{1}{2}\big((\alpha-\beta) + (\alpha+\beta+2)(2t -1)\big)
   = -(1+\beta)+(\alpha+\beta+2)t \,,
 \end{equation}
and then in the special case  $\beta = p-1 $ and  $n=2 $
\begin{equation}\label{degree1special}
\begin{split}
&
   P_1^{0, p-1}(2t -1) = -p + (p+1)t\,, \quad t_p=\frac{p}{p+1}
   \\&
     P_1^{1, p-1}(2t -1) =  -p + (p+2)t\,, \quad s_p=\frac{p}{p+2}
\end{split}
    \end{equation}
    where  $t_p $ resp.  $s_p $ stand for the zero of the corresponding polynomial in the interval  $[0,1] $. By
Theorem~\ref{jacobi} and Theorem~\ref{bi-end} we have
       \begin{align*}
  \boldsymbol{\lambda}\big(\mathcal{H}_{p,1}(\mathbb{S}^{1})\big)
  &
  = (p+2)\,\, \int_{0}^{1}  t^{\frac{p-1}{2}} \,\big| P_{1}^{0,p-1}(2t -1)\big| dt
  \\&
  =(p+2)\,\bigg( \int_{0}^{t_p}  t^{\frac{p-1}{2}} \,(p - (p+1)t) dt + \int_{t_p}^{1}  t^{\frac{p-1}{2}} \,(-p + (p+1)t) \;dt\bigg)
  \,,
\end{align*}
and
\begin{align*}
  \boldsymbol{\lambda}\big(\mathfrak{P}_{p,1}(\mathbb{S}^{1})\big)
  &
  = (p+1)\,\, \int_{0}^{1}  t^{\frac{p-1}{2}} \,\big| P_{1}^{1,p-1}(2t -1)\big| dt
  \\&
  =(p+1)\,\bigg( \int_{0}^{s_p}  t^{\frac{p-1}{2}} \,(p - (p+2)t) dt + \int_{s_p}^{1}  t^{\frac{p-1}{2}} \,(-p + (p+2)t)\;dt\bigg)
  \,,
\end{align*}
Then the two formulas for the projection constants (and as a consequence the claimed limits) follow by  partial integration.
          \end{proof}

We now consider the case where both parameters are equal to one in any dimension  $ n \geq 2  $.

\begin{proposition}\label{caso 1-1}
  For every  $n \ge 2  $,
  \[
  \boldsymbol{\lambda}\big(\mathcal{H}_{1,1}(\mathbb{S}^{n-1})\big)
  \,= \,2 (n+1) \Big(1-\frac{1}{n}\Big)^n\,
  \]
  and
  \[
  \boldsymbol{\lambda}\big(\mathfrak{P}_{1,1}(\mathbb{S}^{n-1})\big)
  \,= \,2 (n+1) \Big(1-\frac{1}{n+1}\Big)^n -1\,.
  \]
\end{proposition}

Note that by Equation~\eqref{dimfor-comp}
\[
\text{
 $\dim \mathcal{H}_{1,1}(\mathbb{S}^{n-1}) = (n+1)(n-1) $
\quad and \quad
 $\dim \mathfrak{P}_{1,1}(\mathbb{S}^{n-1}) = n^2 $\,.
}
\]

\begin{proof}
  Similarly to the proof of Proposition~\ref{nazibude}, this is again a simple consequence of  Theorem~\ref{jacobi} and  Theorem~\ref{bi-end}, using the explicit representation of   $P_1^{n-2,0}(2t-1) $ and  $P_1^{n-1,0}(2t-1) $ given by Equation~\eqref{degree1}.
  \end{proof}

\begin{remark}
    In this paper, the dimension $n$ is kept fixed and we study the asymptotic behavior of the projection constant as one or both of the parameters $p$ and $q$ vary. A complementary viewpoint would be to fix the parameters and let the dimension tend to infinity. As part of a separate research project by the authors, this regime can also be analyzed explicitly, leading to asymptotic formulas expressed in terms of integrals involving Laguerre polynomials.
\end{remark}

\begin{comment}
\subsection{Open questions}

By Proposition~\ref{prop: bound bihom} and \ref{prop: bound harm}, we observe that when  $ q  $ is fixed and  $ p  $ grows, the projection constant remains bounded (by a constant which depends only on  $ n  $ and  $ q  $). We conclude this section with two open problems, leaving room for further exploration.

\begin{problem}
Prove or disprove that for each  $ n \geq 2  $,
\[
\sup_{p,q} \boldsymbol{\lambda}\big(\mathcal{H}_{p,q}(\mathbb{S}^{n-1})\big)
\quad \text{and} \quad
\sup_{p,q} \boldsymbol{\lambda}\big(\mathcal{P}_{p,q}(\mathbb{S}^{n-1})
\]
are bounded by a constant depending only on  $ n  $.
\end{problem}
\end{comment}

\section{Non-compact inclusions: Variants of the Ryll-Wojtaszczyk results}

The motivation for this section stems once again from a significant result in the remarkable work of Ryll and Wojtaszczyk \cite{ryll1983homogeneous}.
In their study, the authors address the question posed by Waigner regarding whether the identity map from  $ H_\infty(B_{\ell_2^n})  $ to  $ H_1(B_{\ell_2^n})  $ is compact for any positive integer  $ n \ge 2  $. They demonstrate the existence of a~sequence  $ (p_k)$ of $k$-homogeneous polynomials on the complex unit ball  $ B_{\ell_2^n}  $ of  $\mathbb{C}^n $ such that  $ \|p_k\|_{\infty} = 1$ and  $ \|p_k\|_2 \geq \sqrt{\pi}\,2^{-n}$.
It is noteworthy that the authors mention in \cite{ryll1983homogeneous} that Waigner's question is closely connected to a well-known problem posed by Rudin in his monograph \cite{rudin1980}: Does there exist an inner function on the open unit ball of the Hilbert space  $ \ell_2^n$ for $ n \ge 2$? For further details on this topic, see \cite{rudin1985} and \cite{wojtaszczyk1996banach}.

We recall the well-known result of Rutovitz \cite{rutovitz} regarding continuous projections in finite dimensional Hilbert spaces. This result establishes that the projection constant of the complex $ n $-dimensional Hilbert space is given by
\begin{align}
\label{Ruto}
\boldsymbol{\lambda}\big(\ell_2^n(\mathbb{C})\big) = \frac{\sqrt{\pi}}{2} \frac{n!}{\Gamma\left(n + \frac{1}{2}\right)}\,.
\end{align}

Now we are ready to prove the following theorem, which is a~general version of the Ryll-Wojtaszczyk result.

\medskip

\begin{theorem} \label{Ryll-Wojtaszczyk- polynomials}
Let  $\big((K,\mu),(G,\mathrm{m}),\varphi\big) $ be a~Rudin triple and  $S $ a~ $\varphi $-invariant, finite dimensional subspace of   $C(K) $.
Then there is  $f \in S $ such that
\[
\text{$\|f\|_\infty = 1 $ \quad and \quad  $\|f\|_2 \ge \frac{\sqrt{\pi}}{2}   \frac{1}{\boldsymbol{\lambda}(S)} $   \,.}
\]
\end{theorem}

\smallskip

\begin{proof}
First we write  $S_2 $ for the finite dimensional Hilbert space defined  by  $S $ considered as a subspace of  $L_2(\mu) $. Of course,
$S $ itself  carries the supremum norm inherited
from $C(K) $. Then the ideal property of the projection constant implies that
\[
\boldsymbol{\lambda}(S_2) \,\leq \,\boldsymbol{\lambda}(S)\, \|\id \colon S \to S_2\| \, \|\id: S_2 \to S\|\,,
\]
and consequently
\[
\|\id: S \to S_2\|
\,\geq \,
\frac{\boldsymbol{\lambda}(S_2) }{\boldsymbol{\lambda}(S) }
\frac{1}{\|\id: S_2 \to S\|}\,.
\]
But by the result of Rutkovitz \eqref{Ruto}
\begin{align*}
\boldsymbol{\lambda} (S_2)
= \frac{\sqrt{\pi}}{2}   \frac{\Gamma(1 + \dim S)}{\Gamma(\frac{1}{2} + \dim S)} \ge \frac{\sqrt{\pi}}{2} \sqrt{\dim S}\,,
\end{align*}
and hence it suffices to show that
\begin{align} \label{claimo}
\|\id \colon  S_2 \to S\| \leq  \sqrt{\dim S}\,.
\end{align}
Indeed, by Lemma~\ref{theorem kernel0} (i) and (vi), the reproducing kernel of $ S $, denoted $k_{S}\colon K \times K \to \mathbb{C}$,
satisfies the following properties for all $ f \in S $ and $ x \in K$:
\[
f(x) = \langle f, k_S(x, \cdot) \rangle_{L_2(\mu)} \quad \text{and} \quad \|k_S(x, \cdot)\|_{L_2(\mu)} = \sqrt{\dim S}\,.
\]
Consequently, for every  $f \in S $
\begin{align*}
\|f\|_\infty = \sup_{x \in K}| \langle f, k_S(x, \cdot) \rangle_{L_2(\mu)}|
\leq  \|f\|_{L_2(\mu)} \sup_{x \in K}\|k_S(x, \cdot)\|_{L_2(\mu)} =  \|f\|_{L_2(\mu)} \sqrt{\dim S}\,.
\end{align*}
This completes the proof.
\end{proof}

\smallskip

As a consequence, the following result follows.

\begin{corollary} \label{ryll compact}
Let $\big((K,\mu),(G,\mathrm{m}),\varphi\big) $ be a~Rudin triple and let  $X$ be an infinite dimensional closed subspace  of \,$C(K)$.
Suppose that $\big(S_k\big)_{k=1}^\infty $ is a~sequence  of finite dimensional $\varphi$-invariant subspaces of  $X $, which are
pairwise orthogonal in  $L_2(\mu) $ and such that  $\gamma:= \sup_{n\geq 1} \boldsymbol{\lambda}(S_k)<+ \infty $. Then the continuous inclusion map
$X \hookrightarrow L_1(\mu) $ is not a~compact operator.
\end{corollary}

\begin{proof}
By Theorem~\ref{Ryll-Wojtaszczyk- polynomials}, it follows that for each  $k\in \mathbb{N} $, we can find  $f_k \in S_k $ such that
$\|f_k\|_\infty=1 $ and $\|f_k\|_2 \ge \frac{\sqrt{\pi}}{2}   \frac{1}{\boldsymbol{\lambda}(S_k)} $. Thus our hypothesis yields
\[
\|f_k\|_2 \ge \frac{\sqrt{\pi}}{2\gamma}, \quad\, k\in \mathbb{N}\,,
\]
and so we get
\[
\frac{\sqrt{\pi}}{2\gamma} \leq \|f_k\|_2 \leq \big(\|f_k\|_1\,
\|f_k\|_\infty\big)^{\frac{1}{2}} = \sqrt{\|f_k\|_1}\,.
\]
Note that $(f_k) $ is a~bounded orthogonal sequence, so  $f_k \to 0 $ weakly in  $L_2(\mu) $ and hence  $f_k\to 0 $ weakly in  $L_1(\mu)$.
Since $\inf_k \|f_k\|_1 >0 $, the required  statement follows.
\end{proof}

For any infinite set $J \subset \mathbb{N}$ and a fixed degree $q$, we consider the subspace of $C(\mathbb{S}^{n-1})$
\[
\mathcal{H}_{J,q} := \overline{\operatorname{span}\big\{ \mathcal{H}_{p,q}(\mathbb{S}^{n-1}) : p \in J \big\}},
\]
which represents bihomogeneous harmonic polynomials, where the homogeneity degree in $z = (z_j) $ lies in the infinite set  $J$, and
the degree in $ \overline{z} = (\overline{z_j}) $ is fixed at  $ q $. As a consequence of Corollary~\ref{ryll compact}, Lemma~\ref{ortogonales},
and Proposition~\ref{prop: bound harm}, we deduce that this space is not compactly embedded in  $ L_1(\mathbb{S}^{n-1}, \sigma_n) $, where
$\sigma_n $ denotes the Haar measure on the sphere  $\mathbb{S}^{n-1} $. This leads to the following result:

\begin{corollary} \label{coro: inclusion no compacta}
Let  $J \subset \mathbb{N} $ be an infinite set, and let  $q$ be a fixed degree. Then, any closed subspace  $X $ of $C(\mathbb{S}^{n-1})$
that contains $ \mathcal{H}_{J,q} $ is not compactly included in  $L_1(\mathbb{S}^{n-1}, \sigma_n)$.
\end{corollary}

We conclude with the observation that, as noted in Remark~\ref{iso piola}, this conclusion also extends to the space $\mathcal{H}_{p,J}$,
where $p$ is the fixed degree (with the analogous definition).

%\section*{Acknowledgments}
%We thank the referee for a careful reading of the manuscript and for valuable suggestions that improved the quality of the exposition.

%During the preparation of this manuscript, the authors used ChatGPT (OpenAI) for language editing and stylistic improvements. The authors reviewed and edited the text and take full responsibility for the content of the article.

\end{document}